\documentclass{AIMS}
\pdfoutput=1
\usepackage{graphicx}

\usepackage{amssymb,amsmath,amsfonts}
\usepackage{url,xcolor}
\usepackage{microtype}
\usepackage{amsmath}
  \usepackage{paralist}
  \usepackage{graphics} 
  \usepackage{epsfig} 
 \usepackage[colorlinks=true]{hyperref}
\hypersetup{urlcolor=blue, citecolor=red}

\def\currentvolume{X}
 \def\currentissue{X}
  \def\currentyear{201X}
   \def\currentmonth{XX}
    \def\ppages{X--XX}

\newtheorem{lemma}{Lemma}
\newtheorem{theorem}[lemma]{Theorem}

\newtheorem{corollary}[lemma]{Corollary}
\theoremstyle{definition}
\newtheorem{definition}[lemma]{Definition}

\renewcommand{\d}{\mathop{}\!\mathrm{d}}
\renewcommand{\i}{\mathrm{i}}
\newcommand{\new}[1]{#1}
\title[Periodic orbits and large delay]{On the stability of periodic
  orbits in delay equations with large delay}

\author[Sieber, Wolfrum, Lichtner and Yanchuk]{}


\email{j.sieber@exeter.ac.uk}
\email{matthias.wolfrum@wias-berlin.de}
\email{mark.lichtner@wias-berlin.de}
\email{yanchuk@math.hu-berlin.de}

\thanks{The authors acknowledge the support of DFG Research Center
  {\sc Matheon} \textquotedblleft{}Mathematics for key
  technologies\textquotedblright{} under the project D21}

\begin{document}
\maketitle

\centerline{\scshape Jan Sieber}
\medskip
{\footnotesize
   \centerline{University of Exeter, UK}
} 

\medskip

\centerline{\scshape Matthias Wolfrum$^1$, Mark Lichtner$^1$ and
  Serhiy Yanchuk$^2$} \medskip {\footnotesize
  \centerline{$^1$Weierstrass Institute for Applied Analysis and
    Stochastics, Berlin, Germany} \centerline{$^2$Humboldt University
    of Berlin, Institute of Mathematics, Berlin, Germany}

}

\bigskip


\begin{abstract}\noindent
  We prove a necessary and sufficient criterion for the exponential
  stability of periodic solutions of delay differential equations with
  large delay. We show that for sufficiently large delay the Floquet
  spectrum near criticality is characterized by a set of curves, which we
  call asymptotic continuous spectrum, that is independent on the
  delay. 
\end{abstract}

\section{Introduction}
\label{sec:intro}
Delay-differential equations (DDEs) are similar to ordinary
differential equations (ODEs) except that the right-hand side may
depend on the past. For example, they could be of the form
\begin{equation}\label{eq:introdde}
  \dot x(t)=f(x(t),x(t-\tau))
\end{equation}
where $x(t)$ is a vector in $\mathbb{R}^n$ and the delay $\tau>0$ 
decides how far one looks into the past. 
When studying DDEs as dynamical systems
one notices that equilibria do not depend on the delay $\tau$. More
precisely, their location and number is independent of
$\tau$. However, their stability changes significantly when one varies
$\tau$, an effect that is well known and of practical importance in
engineering and control \cite{LMNS09,S89}. Exponential stability is
given by the spectrum of the linearization of the DDE in its
equilibrium. This spectrum, in turn, can be expressed as roots of an
analytic function (a polynomial of exponentials
$\lambda\mapsto\exp(-\lambda\tau)$). Lichtner \emph{et al}
\cite{LWY11} classified rigorously which types of limits this spectrum
can have as $\tau$ approaches infinity. Roughly speaking, for sufficiently
large $\tau$ all except maximally $n$ eigenvalues form bands near the imaginary
axis
. After rescaling their real part by $1/\tau$ one finds
that these bands converge to curves, called \emph{asymptotic continuous
  spectrum}. They are given as root curves of parametric polynomials,
and are, thus, much easier to compute than the eigenvalues of the
singularly perturbed large-delay problem. Of practical relevance are
then stability criteria based entirely on the asymptotic spectra that
guarantee the stability of an equilibrium for sufficiently large
delays $\tau$.

This paper gives a similar result for periodic orbits of
\eqref{eq:introdde}. In contrast to equilibria, periodic orbits change
as the delay $\tau$ varies, such that the statement about
``independence'' of the delay (or, rather, re-appearance) has to be
formulated more carefully. Let us look at a two-dimensional example to
illustrate the observation made by Yanchuk \& Perlikowski \cite{YP09}:
\begin{align}
  \label{eq:exx}
  \dot x(t)&=\alpha x(t)-2\pi y(t)-x(t)\left[x(t)^2+y(t)^2\right]\\
  \label{eq:exy}
  \dot y(t)&=2\pi x(t)+\alpha y(t)-y(t)\left[x(t)^2+y(t)^2\right]+y(t-\tau)
\end{align}
where we fix $\alpha\approx-0.1$ and vary
$\tau$. System~\eqref{eq:exx}--\eqref{eq:exy} consists of the normal
form for the supercritical Hopf bifurcation with an additional delayed
term $y(t-\tau)$ in the second equation, which breaks the rotational
symmetry of the instantaneous terms.
\begin{figure}[t]
  \centering
  \includegraphics[width=0.8\textwidth]{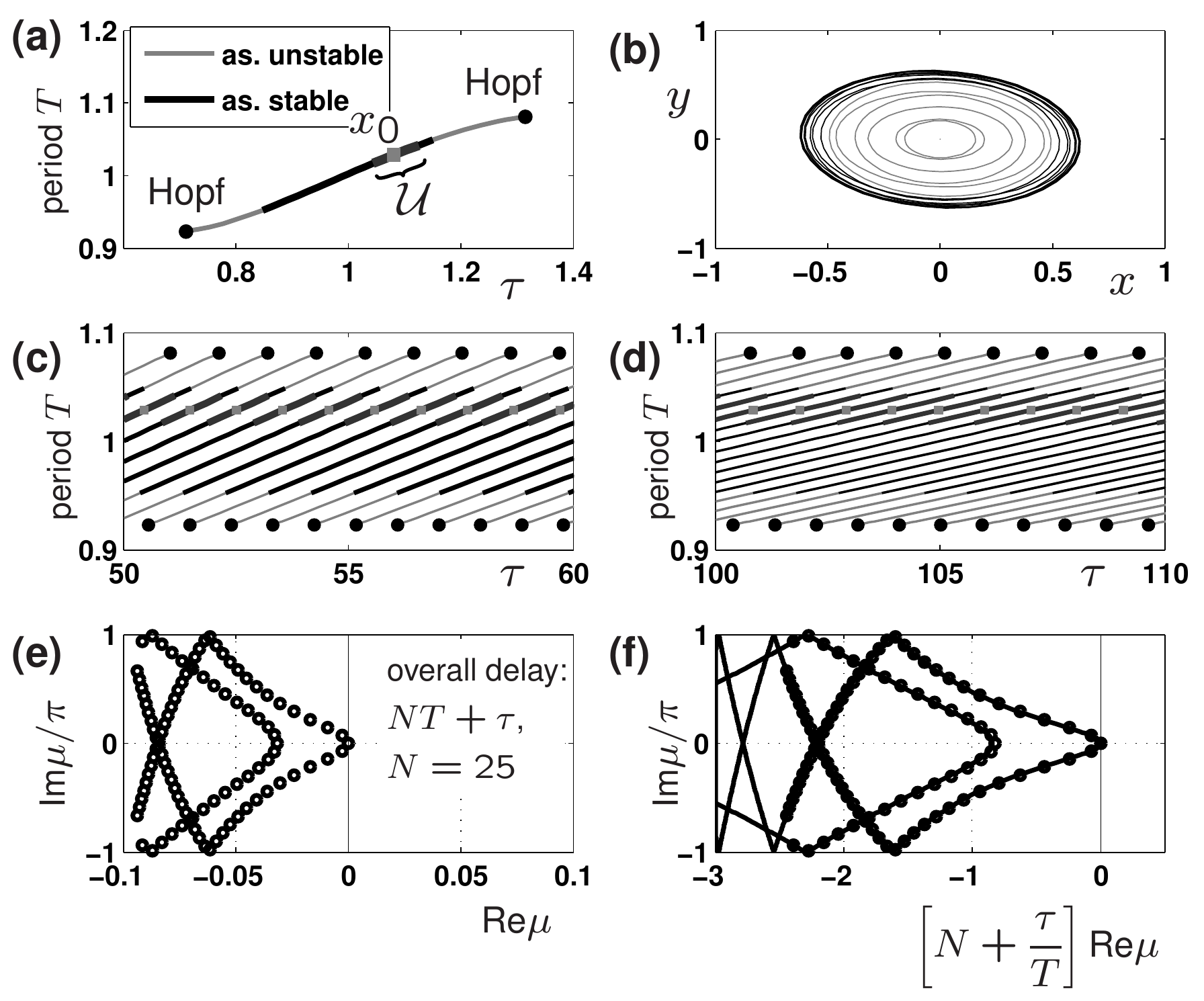}
  \caption{Illustration of stability 
    for the example \eqref{eq:exx}--\eqref{eq:exy} (computed with
    DDE-Biftool \cite{ELS01,RS07,SER02}): (a) bifurcation diagram of
    periodic orbits in the $(\tau,T)$-plane; (b) phase portraits of
    periodic orbits along branch; (c,d) images of the branch under
    transformation $\tau\mapsto NT(\tau)+\tau$, illustrating overlaps
    and coexistence; (e) dominant Floquet exponents of selected
    periodic orbit (grey square in panel (a)); (f) dominant Floquet
    exponents after rescaling of real parts and curves of asymptotic
    continuous spectrum. Parameters: $\alpha=-0.10779$ for (a--d),
    $\tau=1.081$, $N=25$ for (c,d).}
  \label{fig:illu}
\end{figure}
Numerically, one can observe that the system has a family of periodic
orbits for delays between $\tau\approx0.7$ and $\tau\approx1.3$. The
period and the phase portrait projected onto the $(x,y)$-plane of
these orbits are shown in Figure~\ref{fig:illu}(a) and (b). The family
repeats, the orbits keeping their shape, for every integer $N$ if we
change
the horizontal axis in Figure~\ref{fig:illu}(a) to
$NT(\tau)+\tau$. The transformation $NT(\tau)+\tau$ is not just a
parallel shift, since the dependence $T(\tau)$ is, generically,
nontrivial (in the example $\d T/\d \tau>0$).  This leads to an
overlapping of the families for large $\tau$ as shown in
Figure~\ref{fig:illu}(c) and (d). 

Specifically, let us consider a regular periodic orbit $x_0$, existing
for a fixed value $\tau_0$ of $\tau$ (regularity means that the unit
Floquet multiplier of $x_0$ is simple, see
Section~\ref{sec:nontech}). Then $x_0$ persists for $\tau$ in a small
neighborhood of $\tau_0$ such that we have a branch of periodic orbits
depending locally on the parameter $\tau$ close to $\tau_0$. Let
$\mathcal{U}$ be a small neighborhood of this orbit along the
branch. Along this branch the period $T$ of the orbit is a smooth
function $T(\tau)$ of $\tau$. In Figure~\ref{fig:illu}(a) an example
orbit $x_0$ is indicated by a grey square at $\tau_0=1.081$ and its
neighborhood $\mathcal{U}$ is highlighted by a slight thickening of
the line. If we assume that $T'(\tau_0)\neq0$ (which is a genericity
condition) then $T(\tau)$ will have a slope uniformly bounded away
from zero for all $\tau$ in the neighborhood corresponding to
$\mathcal{U}$. This implies that the images of the neighborhood
$\mathcal{U}$ are stretched proportionally to $N$ under the
transformation $\tau\mapsto NT(\tau)+\tau$.  Thus, for any given large
$\tau$ a large number of periodic orbits from $\mathcal{U}$ coexist,
and the number of coexisting orbits is proportional to $\tau$ (see the
overlapping images of $\mathcal{U}$ in Figure~\ref{fig:illu}(c) and
(d)).

The next question to ask is: which of those many coexisting periodic
orbits for large delays are dynamically stable? Is it possible to
derive sharp stability criteria for the large-delay orbits based on
quantities independent of the delay? More precisely, what is the
stability of a given periodic orbit $x_0$ (such as the one indicated
by a grey square in Figure~\ref{fig:illu}(a) for delays
$\tau=NT(\tau_0)+\tau_0$ as $N$ goes to infinity (this would be the
sequence of orbits indicated by grey squares in
Figure~\ref{fig:illu}(c,d)). 

Our example suggests that the stability for $\tau=NT(\tau_0)+\tau_0$
where $N\to\infty$ can indeed be determined from spectral properties
of the periodic orbit at the small delay $\tau_0$.
Figure~\ref{fig:illu}(e) shows the spectrum of the periodic orbit
$x_0$ highlighted by a grey square in Figure~\ref{fig:illu}(a) for
$N=25$. One can see that, first, its Floquet exponents form bands,
and, second, there is a large number of Floquet exponents very close
to the imaginary axis (note the scale of the horizontal axis in
Figure~\ref{fig:illu}(e)). Figure~\ref{fig:illu}(f) illustrates one of
the results of this paper: after rescaling the real part, Floquet
exponents converge to curves for $N\to\infty$. These curves, the
\emph{asymptotic continuous spectrum} are computable by solving
\emph{regular} periodic boundary value problems parametrized by
$\omega$, the vertical axis in Figure~\ref{fig:illu}(f). Since the
original nonlinear DDE is autonomous, one of the curves of the
asymptotic continuous spectrum touches the imaginary axis. Similar to
the equilibrium case, we establish that the asymptotic continuous
spectrum (together with the strongly unstable spectrum, see
Section~\ref{sec:nontech}) determines the stability of the periodic
orbit $x_0$ for sufficiently large $N$. 

We have colored the branch in Figure~\ref{fig:illu}(a) already
according to the conclusions from the asymptotic spectra. The part
that is displayed as a black curve in Figure~\ref{fig:illu}(a) is the
part of the branch that will be exponentially stable as $N$ tends to
infinity, whereas the grey part will be exponentially unstable, having
a large number of weakly unstable Floquet
exponents. 

In short, the observation by Yanchuk \& Perlikowski \cite{YP09}
implies that, if we find a periodic orbit $x_0$ of the DDE
\eqref{eq:introdde} for a fixed small delay $\tau_0$, then (under some
genericity conditions) the DDE \eqref{eq:introdde} has a large number
of similar orbits coexisiting for any sufficiently large delay
$\tau$. This paper establishes a sharp critierion (again under some
genericity conditions) determining if all of these coexisting periodic
orbits are dynamically stable. It does so by providing formulas for
the Floquet exponents of the periodic orbit with delay
$\tau=NT(\tau_0)+\tau_0$ that are independent of $N$ but valid
asymptotically for large $N$.  For the example DDE
\eqref{eq:exx}--\eqref{eq:exy}, our results imply that the system has
a large number of coexisting stable periodic orbits for every
sufficiently large delay $\tau$ (as suggested by the
Figures~\ref{fig:illu}(c,d) where the black parts of the branches are
stable).

Section~\ref{sec:nontech} gives a non-technical overview of the
results gradually developed and proven in the later sections. One
central part of our paper is the construction of a characteristic
function
\begin{displaymath}
  \mu\mapsto h\left(\mu,\exp(-(NT(\tau_0)+\tau_0)\mu)\right)\mbox{,}  
\end{displaymath}
the roots of which are the Floquet exponents of the periodic orbit for
delay $\tau_0+NT(\tau_0)$, and for which we can study the limit
$N\to\infty$. This construction is given in
Section~\ref{sec:charmat}. The existence of this function $h$ permits
us to follow the approach from \cite{LWY11} and to extend their
techniques to the case of periodic solutions.  In the following
sections~\ref{sec:strong} and \ref{sec:pc} we describe two parts of
the Floquet spectrum that show a different scaling behavior for large
$N$ (and, thus, $\tau$). The strongly unstable spectrum, converging to
a finite number of asymptotic Floquet exponents which are determined
by the instantaneous terms, is investigated in
Section~\ref{sec:strong}. Then, in Section~\ref{sec:pc}, we analyze
the Floquet exponents given by the asymptotic continuous spectrum
shown in Fig~\ref{fig:illu}(f).  Based on these results, we can
then prove in Section~\ref{sec:stab} our main result, a criterion
about asymptotic stability based on the location of the asymptotic
continuous and strongly unstable spectrum.

\new{In contrast to the case of spectra at equilibria, where the asymptotic
continuous and strongly unstable spectrum in many cases can be
calculated explicitly (see \cite{LWY11}), the corresponding parts of
the Floquet spectrum of a periodic orbit $x_0$, can typically only be
computed numerically.  This limitation is not specific to our
  results but applies equally to most stability results based on the
  Floquet spectrum for ODEs.  Even the periodic orbit itself is in
most cases only computable with numerical methods. An exception is the
case, where the periodic orbit is at the same time a symmetry orbit of
the system. In this case, examples of asymptotic continuous and
strongly unstable Floquet spectrum have been calculated explicitly in
\cite{WY06,YW10}.

Conversely, the results presented in this paper provide an
  approach to approximate Floquet spectra numerically for large delays
  $\tau$. If one uses numerical methods on problems with large delays,
  one faces the difficulty that the size of the matrix arising in the
  discretized eigenvalue problem grows not only with the desired
  accuracy (which is natural) but also with $\tau$, even if the period
  of the orbit remains bounded. This is the case for the numerical
  methods used in DDE-Biftool \cite{RS07}. This increase is to be
  expected because the number of Floquet exponents close to the
  imaginary axis increases with $\tau$ (see
  Theorem~\ref{thm:main}). In contrast to this, the asymptotic spectra
  can be computed with the same numerical method as in DDE-Biftool and
  a matrix size that is uniformly bounded for large $\tau$.  This
  paper does not discuss the details of the numerical computation of
  asymptotic spectra. However, our construction of the characteristic
  function $h$ is uniform with respect to $\tau$ and could in
  principle be implemented numerically. In practice, it is better to
  apply the same analysis as is done in this paper to the large
  discretized eigenvalue problem.}

In Fig. \ref{fig:illu}(f) we demonstrated the good 
agreement between original Floquet exponents and asymptotic spectra already in the situation
of a moderately large delay, where the original Floquet exponents are still computable.
Moreover, our stability criterion, which can be reliably verfied in this way, is valid independently
on the actual value of the large delay.
\section{Basic concepts and overview of the results}
\label{sec:nontech}
\subsection{Periodic orbits, stability, and Floquet exponents}
Let $x_*$ be a periodic orbit of the $n$-dimensional autonomous
nonlinear delay differential equation (DDE)
\begin{equation}
  \label{eq:gendde}
  \dot x(t)=f(x(t),x(t-\tau))\mbox{,}
\end{equation}
that is, $x_*(t)$ satisfies \eqref{eq:gendde} for all times $t$ and
has period $T$: $x_*(t)=x_*(t+T)$ for all $t\in\mathbb{R}$. Without loss of
generality we may assume that $T=1$ (this can be achieved by a
rescaling of time and the delay $\tau$). If the period $T$ is equal to
$1$ then $x_*$ is also a periodic orbit of
\begin{equation}
  \label{eq:nonlinddeN}
  \dot x(t)=f(x(t),x(t-\tau-N))\mbox{,}  
\end{equation}
where $N$ is a natural number 
and $\tau\in[0,1)$. We denote the restriction of the periodic function
$x_*:\mathbb{R}\mapsto\mathbb{R}^n$ to the interval $[-\tau-N,0]$ also by $x_*$ such
that $x_*$ is an element of $C([-\tau-N,0];\mathbb{R}^n)$.
  

We are concerned with the question whether the periodic orbit $x_*$
is stable or unstable for sufficiently large $N$ in the following
sense:
\begin{definition}[Exponential orbital stability and instability]
  \label{def:stab}
  \ \\ Let $X(t;\cdot)$ be the semiflow on $C([-\tau-N,0];\mathbb{R}^n)$ induced by
  DDE \eqref{eq:nonlinddeN}. 
  The periodic orbit $x_*$ is called \textbf{exponentially orbitally
    stable} if there exists a decay rate $\gamma>0$ such that all
  initial history segments $x_0$ in a
  neighborhood of $x_*$ satisfy
  \begin{displaymath}
    \|X(t;x_0)-X(t+t_0;x_*)\|_\infty\leq 
    C\exp(-\gamma t)\|x_0- x_*\|_\infty
  \end{displaymath}
  for some time shift $t_0$ and some constant $C\geq0$.

  Similarly, $x_*$ is called \textbf{exponentially unstable} if there
  exists a growth rate $\gamma>0$, a neighborhood $\mathcal{N}$ of
  $x_*$ and a constant $C>0$ such that one can find initial
  history segments $x_0\neq x_*$ arbitrarily close to $ x_*$ that satisfy
  \begin{displaymath}
    \|X(n;x_0)-X(n;x_*)\|_\infty\geq
    C\exp(\gamma n)\|x_0- x_*\|_\infty>0
  \end{displaymath}
  for all $n\in\mathbb{N}$ as long as $X(n;x_0)$ stays in the neighborhood
  $\mathcal{N}$.
\end{definition}
This is the standard definition for stability of periodic orbits used
also for ODEs except that the phase space is
$C([-\tau-N,0];\mathbb{R}^n)$ instead of $\mathbb{R}^n$. The notation
$\|\cdot\|_\infty$ refers to the usual maximum norm in
$C([-\tau-N,0];\mathbb{R}^n)$.

 Textbook theory of delay equations reduces the stability
problem to the problem of finding eigenvalues of the linear map
$M_N:C([-\tau-N,0];\mathbb{C}^n) \to C([-\tau-N,0];\mathbb{C}^n)$, which is given as
the time-$1$ map of the linear DDE
\begin{equation}
  \label{eq:lindde}
  \dot x(t)=A(t)x(t)+B(t)x(t-\tau-N)\mbox{,}
\end{equation}
where the time-dependent $n\times n$-matrices $A(t) \in \mathbb{R}^{n
  \times n}$ and $B(t) \in \mathbb{R}^{n \times n}$ are the partial
derivatives of $f$ in $x_*$: $A(t)=\partial_1f(x_*(t),x_*(t-\tau))$
and $B(t)=\partial_2f(x_*(t),x_*(t-\tau))$ \cite{HL93}. If the
right-hand side $f$ of the nonlinear problem \eqref{eq:gendde} is
smooth in its arguments then the matrices $A$ and $B$ are also smooth
periodic function of time $t$ in the interval $[0,1]$.

Since the time-$1$ map $(M_N)^{N+1}$ is 
compact 
the spectral theory for compact operators and the polynomial spectral
mapping theorem imply that the spectrum $\sigma(M_N)$ consists of a
sequence of eigenvalues of finite multiplicity accumulating only at
zero (zero is the only element of $\sigma(M_N)$ that is not in the
point spectrum). Also, $\lambda=1$ is always an eigenvalue of $M_N$
because $\dot x_*(t)$ satisfies \eqref{eq:lindde} and has period $1$.

The periodic orbit $x_*$ is exponentially (\emph{orbitally}) stable if
and only if
\begin{itemize}
\item[(1)] the eigenvalue $1$ of $M_N$ is algebraically simple, and
\item[(2)] all other eigenvalues of $M_N$ have modulus less than $1$.
\end{itemize}
Similarly, it is exponentially unstable if at least one eigenvalue has
modulus greater than $1$.

Thus, the stability of iterations of $M_N$ is determined by its
eigenvalues.  We also use the term orbitally stable for the map $M_N$,
meaning that $M_N$ satisfies both of the above conditions. 

The state space of  DDE \eqref{eq:lindde} is the function space
$C([-\tau-N,0];\mathbb{R}^n)$.
Thus, initial value problems for
\eqref{eq:lindde} require specifying an infinite-dimensional initial
condition. Similarly, one expects that a boundary value problem for a
DDE requires the specification of infinitely many boundary
conditions. However, \emph{periodic boundary value problems} are
easier to formulate: for example, a solution of the periodic
boundary value problem for the general DDE~\eqref{eq:gendde} for
period $T=1$ is simply a function $x\in C^1([-1,0];\mathbb{R}^n)$ satisfying
\begin{align}
    \dot x(t)&=f(x(t),x(t-\tau)_{\mathrm{mod}[-1,0]})    \label{eq:genbvp}\\
    x(-1)&=x(0)\label{eq:genperbc}
\end{align}
where the notation $(t-\tau)_{\mathrm{mod}[-1,0]}$ stands for $t-\tau+k$ if
$t-\tau\in[-k-1,-k)$ and $k\in\mathbb{Z}$ is the integer part of $\tau-t$.
Since $x$
is continuous and satisfies the periodicity condition
\eqref{eq:genperbc} it can be extended continuously to a continuous
function on the whole real line by defining
$x(s)=x(s_{\mathrm{mod}[-1,0]})$. Consequently, the right-hand side of
\eqref{eq:genbvp} is continuous for all $t\in[-1,0]$, which guarantees
that $x$ can really satisfy the differential equation pointwise and is
an element of $C^1$. The solution $x$ then automatically satisfies
$\dot x(-1)=\dot x(0)$, and, thus, by induction is as smooth as the
right-hand side $f$. In this respect, periodic boundary value problems
for DDEs are similar to boundary value problems for ODEs. In
Section~\ref{sec:charmat} we will reduce linear periodic
boundary value problems for DDEs to low-dimensional linear systems of
algebraic equations.

\begin{definition}\label{def:Floq}
  We call $\mu$ a \emph{Floquet exponent} of $M_N$, and write
  $\mu\in\Sigma_N$, if $\exp(\mu)$ is an eigenvalue of $M_N$.
\end{definition}
Floquet exponents of $M_N$ can be found as those complex numbers $\mu$
for which the periodic boundary value problem
\begin{align}
  \label{eq:perbvp}
  \dot y(t)&=[A(t)-\mu I ]y(t)+\exp(-(N+\tau)\mu)\,B(t)\,
  y((t-\tau)_{\mathrm{mod}[-1,0]})
  \\
  y(0)&=y(-1)\label{eq:perbcond}
\end{align}
has a nontrivial solution $y \in C^1([-1,0], \mathbb{C}^n)$ \cite{HL93}.   
Note that in \eqref{eq:perbvp} we use the delay $\tau\in[0,1)$ to calculate the 
Floquet exponents of $M_N$, i.e. for a periodic orbit of system (\ref{eq:lindde}) with delay $\tau +N$.
Only the factor $\exp(-N\mu)$ in front of $B(t)$ accounts for the large delay
whereas we have just extended $y$ periodically for arguments
less than $-1$.  If \eqref{eq:perbvp} has a non-trivial solution
$y(t)$ for $\mu\in\mathbb{C}$ then it has the non-trivial solution
$y(t)\exp(2\pi \i kt)$ for $\mu+2k\pi \i$ for any integer $k$. Hence, we choose
the Floquet exponent $\mu$ such that its imaginary part is between
$[-\pi,\pi)$.

\subsection{Asymptotic spectra for $N\to\infty$}
The set of Floquet exponents, $\Sigma_N$, forms a discrete subset of
the complex plane, the point spectrum of exponents, which depends on
$N$. In order to describe in which form $\Sigma_N$ has a limit for
$N\to\infty$ we introduce two asymptotic spectra, which are also
subsets of the complex plane. The notation follows \cite{YP09}.

\begin{definition}[Instantaneous and strongly unstable spectrum]
$\mbox{}$
  \label{def:inst:strong}
  The set $\Sigma_A$ of all $\mu\in\mathbb{C}$ for which the linear ODE
  boundary value problem on $[-1,0]$
  \begin{align}
    \label{eq:str:ode}
    \dot y&=[A(t)-\mu I ]y\\
    \label{eq:str:bcond}
    y(0)&= y(-1)
  \end{align}
  has a non-trivial solution $y\in C^1([-1,0];\mathbb{C}^n)$ is called the
  \textbf{instantaneous spectrum}. The subset
  $\mathcal{A}_+\subseteq\Sigma_A$ of those $\mu$ with positive real part
  is called the \textbf{strongly unstable asymptotic spectrum}.
\end{definition}
The instantaneous spectrum $\Sigma_A$ contains exactly $n$ elements
with imaginary part in $[-\pi,\pi)$, counting algebraic multiplicity.
We note that $\Sigma_A$ and $\mathcal{A}_+$ do not depend on $N$ but only
on $A$.  One result of our paper is that all Floquet exponents of
$M_N$ with a real part that is positive uniformly in $N$ converge to
elements of the strongly unstable spectrum $\mathcal{A}_+$.

Yanchuk \& Perlikowski \cite{YP09} observed that the presence of
strongly unstable spectrum is not the only possible cause of
instability for large $N$. They observed that large numbers of Floquet
exponents form bands that have a distance of order $1/N$ from the
imaginary axis and have a spacing of order $1/N$ along the imaginary
axis. In the limit $N\to\infty$ these bands form curves after a rescaling of the
real part by $N$. The limiting curves, called \emph{asymptotic
  continuous spectrum} in \cite{YP09}, were defined with the help of a
parametric periodic boundary value problem:
\begin{definition}[Asymptotic continuous spectrum]
  The complex number $\gamma + i \omega \in\mathbb{C}$ ($\gamma\in\mathbb{R}$, $\omega
  \in [-\pi, \pi)$) lies in the \textbf{asymptotic continuous
    spectrum}, $\mathcal{A}_c$, if the boundary value problem on $[-1,0]$
  \begin{align}
    \label{eq:pc:ode}
    \dot y(t)&=[A(t)-i\omega I ]y(t)+\exp(-\gamma-i \varphi)B(t)
    y((t-\tau)_{\mathrm{mod}[-1,0]})
    \\
    y(0)&=y(-1)\label{eq:pc:bcond}
  \end{align}
  has a non-trivial solution $y\in C^1([-1,0]\;R^n)$ for some
  $\varphi\in\mathbb{R}$. The quantity $\varphi$ is called the \textbf{phase}
  corresponding to $\gamma+i\omega$.
\end{definition}
Again, the asymptotic continuous spectrum does not depend on $N$ but
only on $A$, $B$ and $\tau$.

\subsection{A characteristic function for Floquet exponents}
We will reduce now the study of the various
spectra and their relations to each other to a root-finding problem of
a holomorphic function. The following lemma states the existence of a characteristic function for Floquet exponents
that at the same time can be used to describe the asymptotic spectra $\mathcal{A}_+$ and $\mathcal{A}_c$.
\begin{lemma}[Characteristic function]
  \label{thm:charinformal}
  There exists a function
  $h:\Omega_1\times\Omega_2\subseteq\mathbb{C}\times\mathbb{C}\mapsto\mathbb{C}$ which is
  holomorphic in both arguments with the following properties:
  \begin{enumerate}
  \item\label{thm:MN} $\mu$ is a Floquet exponent of $M_N$
    , i.e. $\mu\in\Sigma_N$, if and only if
    \begin{equation}\label{eq:inf:sigmaN}
      h(\mu,\exp(-(N+\tau)\mu)=0\mbox{,}
    \end{equation}
  \item\label{thm:sigmas} $\mu$ is in the instantaneous spectrum
    $\Sigma_A$ if and only if
    \begin{equation}\label{eq:inf:sigmaA}
      h(\mu,0)=0\mbox{,}
    \end{equation}
    and, hence, $\mu$ is in the strongly unstable asymptotic spectrum
    $\mathcal{A}_+$ if
    \begin{displaymath}
      h(\mu,0)=0\mbox{\quad and\quad}
      \operatorname{Re}\mu>0\mbox{,}
    \end{displaymath}
  \item\label{thm:acs} $\mu=\gamma+i\omega$ is in the asymptotic
    continuous spectrum $\mathcal{A}_c$ if, for some phase $\varphi\in\mathbb{R}$,
    \begin{displaymath}
      h(i\omega,\exp(-\gamma-i\varphi))=0\mbox{.}
    \end{displaymath}
  \end{enumerate}
  The algebraic multiplicity of $\mu$ as a Floquet exponent in the
  statements \ref{thm:MN} and \ref{thm:sigmas} equals its multiplicity
  as a root in \eqref{eq:inf:sigmaN} and \eqref{eq:inf:sigmaA}.
\end{lemma}
From property~\ref{thm:acs} the motivation behind the name asymptotic
\emph{continuous} spectrum becomes clear: if we have a value
$\gamma_0+i\omega_0\in\mathcal{A}_c$ and the corresponding phase
$\varphi_0$, and $\partial_2h(i\omega_0,\exp(-\gamma_0-i\varphi_0))$ is
non-zero (which is generically the case) then a whole curve
$\gamma(\omega)+i\varphi(\omega)$ satisfies
$h(i\omega,\exp(-\gamma(\omega)-i\varphi(\omega)))=0$ for
$\omega\approx\omega_0$. These curves are the bands of the asymptotic
continuous spectrum. Note that from the existence of the trivial Floquet exponent $\mu=0$ we can conclude that $h(0,1)=0$, which in turn implies that $\gamma=\omega=0$ with phase $\varphi=0$ is in $\mathcal{A}_c$. In the generic case where the trivial exponent is contained in a single curve we call it the {\em critical branch} of $\mathcal{A}_c$.

The details of the construction of $h$, which modifies the general
characteristic matrices and functions for periodic delay equations
from \cite{SS11,SSH06}, will be given in section~\ref{sec:charmat}.
The general characteristic function constructed by \cite{SS11,SSH06}
may have poles in the complex plane.  The modification in
Section~\ref{sec:charmat} ensures that these poles of $h(\mu,z)$
stay in the left half-plane.  Hence, the domain
$\Omega_1\times\Omega_2$ of $h$ contains all $\mu$ and $z$ satisfying
$\operatorname{Re}\mu>-R$ and $|z|<\exp(R)$ where $R>0$ is arbitrary
but has to be chosen a-priori. Accordingly, statements
\ref{thm:MN}--\ref{thm:acs} of Lemma~\ref{thm:charinformal} are valid
only if both arguments of $h$ satisfy their respective
restriction. However, this is the case in 
the region of
interest for stability and bifurcations.

The introduction of the characteristic function $h$ clarifies how the
different spectra can be calculated and reduces the analysis of the
spectra to a root-finding problem. After defining $h$ properly one could
even use $h$ to \emph{define} the corresponding spectra by the
properties listed in Lemma~\ref{thm:charinformal}.

\subsection{Main results}
With the help of the 
asymptotic spectra $\mathcal{A}_+$ and $\mathcal{A}_c$ we can formulate now a sharp criterion for the exponential
orbital stability and instability of $M_N$, which is our main result:
\begin{theorem}[Stability/Instability]\label{thm:main}
  The map $M_N$ (and, hence, the periodic orbit $x_*$) is
  \textbf{exponentially orbitally stable} for all sufficiently large
  $N$ if all of the following conditions hold:
  {\renewcommand{\labelenumi}{S-\arabic{enumi}}
    \renewcommand{\theenumi}{S-\arabic{enumi}}
    \begin{enumerate}
    \item \label{thm:instnc} \textbf{(No strong instability)} all elements
      of the instantaneous spectrum $\Sigma_A$ have negative real part
      (this implies in particular that the strongly unstable spectrum
      $\mathcal{A}_+$ is empty), and
    \item \label{thm:zeroreg} \textbf{(Non-degeneracy)}
      $\partial_2h(0,1)\neq0$, that is, the Floquet exponent $0$ is
      simple for sufficiently large $N$, and
    \item \label{thm:acsneg} \textbf{(Weak stability)} except for the
      point $\mu=0$ with phase $\varphi=0$ the asymptotic continuous
      spectrum $\mathcal{A}_c$ is contained in $\{z\in\mathbb{C}:\operatorname{Re} z<0\}$.
    \end{enumerate}
  }
  The map $M_N$ is \textbf{exponentially unstable} for all
  sufficiently large $N$ if one of the following conditions holds
  {\renewcommand{\labelenumi}{U-\arabic{enumi}}
    \renewcommand{\theenumi}{U-\arabic{enumi}}
    \begin{enumerate}
    \item \label{thm:strongunst} \textbf{(Strong instability)} the
      strongly unstable spectrum is non-empty, or
    \item \label{thm:acspos} \textbf{(Weak instability)} a non-empty
      subset of the asymptotic continuous spectrum $\mathcal{A}_c$ has
      positive real part.
    \end{enumerate}
  }
\end{theorem}
The weak stability condition~\ref{thm:acsneg} is equivalent to stating
that for all $\omega\in[-\pi,\pi)$ the function $z\mapsto
h(i\omega,\exp(-z))$ has no roots with non-negative real part (with
the exception of $z=0$ for $\omega=0$). Similarly, the weak
instability condition~\ref{thm:acspos} is equivalent to stating that
$h(i\omega,\exp(-z))=0$ for some $\omega\in[-\pi,\pi)$ and some
$z\in\mathbb{C}$ with positive real part.

Several additional corollaries follow from our analysis:

\subsection*{Decay rate and dominant frequency}
If $M_N$ satisfies \ref{thm:instnc}--\ref{thm:acsneg} and is, thus,
exponentially stable for all sufficiently large $N$ then the decay
rate is at most (and generically) of order $O(N^{-3})$ and the
dominant relative frequency is of order $O(N^{-1})$. That is, the
dominant non-trivial Floquet exponents are a complex pair of the form
$[-CN^{-3}+O(N^{-4})]\pm \i [2\pi/(N+\tau)+O(N^{-2})]$. One branch of the asymptotic
continuous spectrum touches the imaginary axis in a (generically
quadratic) even-order tangency at $0$, and the dominant non-trivial
Floquet exponent lies on this branch and next to the tangency.

\subsection*{Robustness}
A map $M_N$ that satisfies \ref{thm:instnc}--\ref{thm:acsneg} and is,
thus, exponentially stable for sufficiently large $N$ according to
Theorem~\ref{thm:main} remains exponentially stable for all
sufficiently large $N$ under all perturbations to $A$, $B$ and $\tau$
of size less than some $\epsilon>0$. This $\epsilon$ does not depend
on $N$ because the quantities determining the exponential stability do
not depend on $N$ but only on $A$, $B$ and $\tau$. This means that
periodic orbits that are stable for large $N$ are uniformly robust
with respect to perturbations of the system despite their weak
attraction rate of order $N^{-3}$. 

This permits the conclusion about the coexistence of \emph{stable}
periodic orbits in the example in Figure~\ref{fig:illu}. The
asymptotic spectra of orbits in the small neighborhood $\mathcal{U}$
are small perturbations of the asymptotic spectrum shown in
Figure~\ref{fig:illu}(f). In particular the curve touching the
imaginary axis will also be curved to the \emph{left} for all periodic
orbits in $\mathcal{U}$.

\subsection*{Spectral approximation}
Let the instantaneous spectrum $\Sigma_A$ have a positive distance to
the imaginary axis. Then we can guarantee that certain points $z$ in
the complex plane are in the \emph{resolvent set} of $M_N$ (that is,
the periodic boundary value problem
\eqref{eq:perbvp}--\eqref{eq:perbcond} has only the trivial solution
for $\mu=z$):
\begin{itemize}
\item every point $z$ in the positive half-plane that is not in the
  strongly unstable spectrum $\mathcal{A}_+$ is in the resolvent set of
  $M_N$ for sufficiently large $N$.
\item All points of the form $\gamma/(N+\tau)+i\omega$ are in the
  resolvent set of $M_N$ for sufficiently large $N$ if the point
  $z=\gamma+i\omega$ is not in the asymptotic continuous spectrum
  $\mathcal{A}_c$.
\end{itemize}
These two statements about the resolvent set of $M_N$ imply that the
spectrum of $M_N$ must be close to $\mathcal{A}_+$ or (after rescaling)
close to $\mathcal{A}_c$. The other direction also holds if the
instantaneous spectrum $\Sigma_A$ is not on the imaginary axis:
\begin{itemize}
\item if $\mu\in\mathcal{A}_+$ has multiplicity $k$ then $k$ Floquet
exponents of $M_N$ converge to $\mu$ for $N\to\infty$ (counting
multiplicity).
\item Let $\gamma_*+i\omega_*$ be in the asymptotic continuous
  spectrum $\mathcal{A}_c$, that is,
  $$h(i\omega_*,\exp(-\gamma_*-i\varphi_*))=0$$ for some phase
  $\varphi_*\in[-\pi,\pi)$. Then we will find Floquet exponents $\mu_N$
  of $M_N$ that satisfy (note that this is small-$o$)
  \begin{equation}\label{eq:genapprox}
    \operatorname{Re}\mu_N-\frac{\gamma_*}{N+\tau}=o(N^{-1})\mbox{,\quad}
    \operatorname{Im}\mu_N-\omega_*=o(1)\mbox{.}
  \end{equation}
  Estimate \eqref{eq:genapprox} is rather weak. We need and prove a
  stronger and more detailed estimate under the additional condition
  that $\partial_2h(i\omega_*,\exp(-\gamma_*-i\varphi_*))\neq0$. Then we
  have a local root curve $\tilde\gamma(\omega)+i\tilde\varphi(\omega)$
  satisfying
  $h(i\omega,\exp(-\tilde\gamma(\omega)-i\tilde\varphi(\omega)))=0$ for
  $\omega$ near $\omega_*$, and for sufficiently large $N$ we find
  algebraically simple Floquet exponents $\mu_k$ of $M_N$ satisfying
  \begin{align}
    \label{eq:omega1}
    \operatorname{Im}\mu_k&=\frac{2k\pi}{N+\tau}+\frac{1}{N+\tau}
    \tilde\varphi\left(\frac{2k\pi}{N+\tau}\right)+O\left((N+\tau)^{-2}\right)\\
    \label{eq:gamma1}
    \operatorname{Re}\mu_k&=\left[1+O\left((N+\tau)^{-1}\right)\right]
    \frac{\tilde\gamma(\operatorname{Im}\mu_k)}{N+\tau}
    \mbox{,}
  \end{align}
  where $k$ are integers such that $2k\pi/(N+\tau)$ is near $\omega_*$, and
  the $O$-s are smooth functions of $\omega$. These Floquet exponents
  $\mu_k$ form a band of discrete complex numbers approximating the
  curve $\tilde\gamma(\omega)+i\tilde\varphi(\omega)$ of asymptotic
  continuous spectrum.
\end{itemize}

\subsection*{Degeneracies}
Theorem~\ref{thm:main} is sharp except for several degenerate cases
that are excluded by the conditions \ref{thm:instnc}--\ref{thm:acsneg}
and \ref{thm:strongunst}--\ref{thm:acspos}.  Degeneracies limiting the region of stable periodic orbits are:
\begin{itemize}
\item An element of the instantaneous spectrum $\Sigma_A$ has zero
  real part. 
\item The partial derivative $\partial_2h(0,1)$ equals $0$. In this
  case the periodic orbit cannot be exponentially stable.
\item \textbf{(Turing and long wavelength instability)} The
  asymptotically continuous spectrum $\mathcal{A}_c$ touches the
  imaginary axis in a quadratic tangency at some value $\pm
  i\omega_0$. In this case the orbit is not exponentially unstable and
  it is still possible that the periodic orbit is exponentially stable
  for all $N$ (this is different from the stationary case discussed in
  \cite{LWY11}) but it may also be weakly stable or unstable. A
  special case is that $\mathcal{A}_c$ touches the imaginary axis in a
  quadratic tangency in the point $\mu=0$ with phase $\varphi=\pi$. In
  this case the periodic orbit is still exponentially stable for large
  $N$.
\item\textbf{(Modulational instability)} The critical branch of the asymptotic continuous
  spectrum $\mathcal{A}_c$, containing the trivial exponent $\mu=0$, has a tangency with the 
  imaginary axis at $\omega=0$
  that is of higher order than quadratic. In this case the periodic
  orbit is still exponentially stable according to
  Theorem~\ref{thm:main}.
\end{itemize}
The characteristic function $h$ constructed in
Section~\ref{sec:charmat} is still valid for the degenerate cases.
However, for a detailed discussion of these degeneracies, one needs
not only to specify a defining equation for each degeneracy (this is
straightforward) but one also has to state secondary non-degeneracy
conditions.  In particular the case of instantaneous spectrum with
zero real part is somewhat subtle, even though it looks similar to the
other listed degeneracies of co-dimension 1 at first sight. In the
analogous situation for the spectrum of equilibria, it turns out that
instantaneous spectrum $i\omega_0$ with zero real part generically
implies a singularity of the asymptotic continuous spectrum,
$\gamma(\omega_0)=\infty$, and hence is not part of the stability
boundary.  Due to these reasons, we believe that a comprehensive
treatment of the degeneracies is beyond the scope of the present
paper.

\section{Construction of the characteristic matrix and function}
\label{sec:charmat}

Let us choose a constant $R>0$ arbitrarily large.  We construct a
function $h_N(\mu)$ that is analytic for all $\mu$ satisfying
$\operatorname{Re}\mu\geq -R/(N+\tau)$, and that satisfies
$h_N(\mu)=0$ if and only if $\mu$ is a Floquet exponent of the map
$M_N$. Hence, this function can be used to find all Floquet exponents
$\mu$ of $M_N$ that satisfy $\operatorname{Re}\mu \geq -R/(N+\tau)$.  Since
we are interested in the stability of the origin under iterates of
$M_N$, finding the roots of $h_N$ will then be sufficient.

We introduce the complex variable $z \in \mathbb{C}$ and consider the
periodic boundary value problem for $t \in [-1,0]$
\begin{align}
  \label{eq:zbvp}
  \dot y(t)&=[A(t)-\mu I ]y(t)+z\,B(t)\,
  y((t-\tau)_{\mathrm{mod}[-1,0]}),
  \\
  y(0)&=y(-1). \label{eq:zbc}
\end{align}
In a first step we will construct a characteristic matrix
$\Delta(\mu,z)$ for \eqref{eq:zbvp}--\eqref{eq:zbc} such that the
roots of its determinant $h(\mu,z)=\det \Delta(\mu,z)$ will be
precisely those pairs of points $(\mu, z)$ in some subdomain of
$\mathbb{C} \times \mathbb{C}$ for which
\eqref{eq:zbvp}--\eqref{eq:zbc} has a nontrivial continuously
differentiable solution $y \in C^1([-1,0];\mathbb{C}^n)$.  Thus, by
inserting $z=\exp(-(N+\tau) \mu)$ we will then obtain the
characteristic function $h_N(\mu)$ in a subdomain of $\mathbb{C}$.

Consider a
partition of the periodicity interval $[-1,0]$ into $k$ intervals of size $1/k$:
\begin{equation}\label{eq:jidef}
  I_j=\left[t_j,t_{j+1}\right)=
  \left[-1+\frac{j}{k},-1+\frac{j+1}{k}\right)\mbox{\ for $j=0,\ldots,k-1$.}  
\end{equation}
Using this partition we formulate a multiple initial value
problem (MIVP) for a vector of $k$ initial (or restart) values
$(v_0,\ldots,v_{k-1})^T\in\mathbb{C}^{nk}$ (similar to \emph{multiple
  shooting}):
\begin{align}
  \label{eq:ivpext}
  \dot y(t)&=[A(t)-\mu I ]y(t)+
  z B(t)y((t-\tau)_{\mathrm{mod}[-1,0]})\\
  \label{eq:icext}
  y(t_i)&=v_i\quad\mbox{for $i=0,\ldots,k-1$}
\end{align}
where $t\in[-1,0]$ and $z\in\mathbb{C}$. Notice that in \eqref{eq:ivpext} the
solution $y$ on the interval $[t_j,t_{j+1})$ depends on the solution
$y(t)$ in other intervals due to the term $y(t-\tau)_{\mathrm{mod}[-1,0]}$ in
the right-hand side of \eqref{eq:ivpext}. 

The main purpose of this partition is to reduce the length of the
integration interval at the cost of increasing the dimension of the
system.  Indeed, this construction is very similar to the construction
in \cite{SW06}, where in the case of rational $\tau$ a reduction to an
equivalent system of ODEs could be achieved, and for the case of
irrational $\tau$ corresponding rational approximations have been
considered.  This can be seen by introducing $u_j(t):=y(t+t_j)$, which
satisfy the system of equations
\begin{align}
\label{eq:uj}
\dot u_j(t)&=[A(t+t_j)-\mu I ]u_j(t)+
z B(t+t_j)u_{m(j)}((t + t_j - t_{m(j)} - \tau)_{\mathrm{mod}[-1,0]})\\
\label{eq:uj0}
u_j(0)&=v_j
\end{align}
with $j=0,\ldots,k-1$.  This system can now be considered as an
initial value problem on the interval $[0,1/k)$, and for a solution on
this smaller interval $[0,1/k)$ each component $u_j(t)$ represents the
solution $y(t+t_j)$ in the corresponding subinterval $I_j$. Note that
instead of the delayed term $u_j((t-\tau)_{\mathrm{mod}[-1,0]})$ we
have inserted a coupling to another component $u_{m(j)}((t + t_j -
t_{m(j)} - \tau)_{\mathrm{mod}[-1,0]})$, where the index $m(j)$ is
chosen such that the argument $(t_j - t_{m(j)} -
\tau)_{\mathrm{mod}[-1,0]}$ is in the interval $[0,1/k)$.  The initial
value problem for the coupled system \eqref{eq:uj}--\eqref{eq:uj0} is
an equivalent formulation of the multiple initial value problem
\eqref{eq:ivpext}--\eqref{eq:icext} and, in a similar way as the
boundary value problem \eqref{eq:zbvp}--\eqref{eq:zbc}, contains both
delayed and advanced arguments. 
\new{If $\tau$ is rational then \eqref{eq:uj}--\eqref{eq:uj0} is a
  system of ODEs. This was the starting point for the rational
  approximations in \cite{SW06}. We take another route by showing
  directly that the Picard-Lindel{\"o}f iteration for
  \eqref{eq:uj}--\eqref{eq:uj0} converges, similar to initial-value
  problems of ODEs.}
  In the sequel, we will clarify this point, using the
formulation \eqref{eq:ivpext}--\eqref{eq:icext} which is more
convenient for our purposes.

Let us denote by $U(t,s,\mu)\in\mathbb{R}^{n\times n}$ the propagation matrix
of the linear ODE defining the instantaneous spectrum,
\eqref{eq:str:ode}. That is,
\begin{align*}
  U(t,s,\mu)v&=y(t) \mbox{\quad where}\\
  y(s)&=v\mbox{\quad and}\\
  \dot y(r)&=[A(r)-\mu I ]y(r)\mbox{\quad for all $r\in[s,t]$.}
\end{align*}
The norm of $U(t,s,\mu)$ can be estimated by
\begin{equation}\label{eq:uest}
  \|U(t,s,\mu)\|_\infty\leq 
  \exp\left(\left[\|A\|_\infty-
      \operatorname{Re}\mu\right](t-s)\right)\mbox{.}
\end{equation}
In 
\eqref{eq:uest} we have used the notation
$\|A\|_\infty=\max_{t\in[-1,0]}\|A(t)\|_\infty$. We will use the same
notation for $B$.  In order to clarify in which sense system
\eqref{eq:ivpext}--\eqref{eq:icext} is an initial value problem and
what it means for $y$ to be a solution of system
\eqref{eq:ivpext}--\eqref{eq:icext} we formulate an integral equation
which is equivalent to
\eqref{eq:ivpext}--\eqref{eq:icext}:
\begin{align}
  \label{eq:varpext}
  y(t)&=[S(\mu)v](t)+z\int\limits_{a_k(t)}^tU(t,s,\mu)
  B(s)y((s-\tau)_{\mathrm{mod}[-1,0]}){\mathrm{d}} s\\
  \label{eq:varpextv}
  [S(\mu)v](t)&=U(t,t_j,\mu)v_j\mbox{\quad if $t\in I_j$, }\\
  \label{eq:varpexta}
  a_k(t)&=t_j\mbox{\quad if $t\in I_j$.}
\end{align}
We note that $S(\mu)v$ and $a_k$ are piecewise continuous functions on
$[-1,0]$ ($[S(\mu)v](t)\in\mathbb{C}^n$ and $a_k(t)\in\mathbb{R}$). They are
continuous on each sub-interval $I_j=[t_j,t_{j+1})$ but have jumps at
the times $t_j$. The integral equation \eqref{eq:varpext} is a
fixed-point problem for $y$. If we find a fixed point $y$ then $y$ may
have discontinuities at the times $t_j$. Thus, the appropriate space
in which to look for solutions of the fixed point problem
\eqref{eq:varpext} is the space of \emph{piecewise continuous}
functions with the usual $\max$-norm $\|y\|_\infty$:
\begin{align}
  C_k=&\begin{aligned}[t] \{&y:[-1,0]\mapsto\mathbb{C}^n:\mbox{\ $y$
      continuous on each
      subinterval $I_j=[t_j,t_{j+1})$}\\
    & \mbox{\ ($j=0,\ldots,k-1$) and $\lim_{t\nearrow t_j} y(t)$
      exists for all $j=1\ldots k$.}\}
  \end{aligned}\label{eq:ckdef}
\end{align}
The right-hand side of the integral equation \eqref{eq:varpext} is an
affine map, mapping $C_k$ back to itself such that \eqref{eq:varpext}
is of the form
\begin{equation}
  \label{eq:varpextmap}
  y=S(\mu)v+zL_k(\mu)y
\end{equation}
where $S(\mu):\mathbb{C}^{nk}\mapsto C_k$ is defined by \eqref{eq:varpextv}, and
$L_k(\mu)y$ is the linear part of the right-hand side in
\eqref{eq:varpext} (that is, the integral term). The linear map
$L_k(\mu):C_k\mapsto C_k$ is continuously differentiable (and, thus,
holomorphic) with respect to the complex variable $\mu$. A simple
estimate for the norm of $L_k$ with respect to the
$\|\cdot\|_\infty$-norm gives us the unique solvability of the fixed
point problem \eqref{eq:varpextmap} (which is actually the integral
equation \eqref{eq:varpext}--\eqref{eq:varpexta}):
\begin{lemma}[Existence and Uniqueness of solutions for IVP]\
\\
  \label{thm:ivp}
  Let $R>0$ be arbitrary. If we set the number of sub-intervals, $k$,
  such that
  \begin{equation}\label{eq:kbound}
    k>C(R):=\max\left\{\|A\|_\infty+R,\|B\|_\infty\exp(1+R)\right\}
  \end{equation}
  then the affine integral equation
  \eqref{eq:varpext}--\eqref{eq:varpexta} has a unique solution $y\in
  C_k$ for all $\mu$ and $z$ satisfying
  \begin{equation}\label{eq:mubound}
    \operatorname{Re}\mu\geq-R\mbox{,\qquad}|z|\leq \exp(R)
  \end{equation}
  and all tuples $(v_0,\ldots,v_k)^T\in\mathbb{C}^{nk}$.
\end{lemma}
The details of the norm estimate for $L_k$ are in
Appendix~\ref{sec:ivpproof}.  The solution $y$ can be written as
$y=[ I -zL_k(\mu)]^{-1}S(\mu)v$. When is this solution $y$ (which is
only in $C_k$ for general $v$) continuously differentiable on the
whole interval?  This requirement is a linear condition on the tuple
$v$.  Let us fix a constant $R>0$ and choose the number of
sub-intervals $k>C(R)$. 
\begin{definition}[Characteristic matrix and
  function] \label{def:gencharmfunc}\ \\ We define the
  \textbf{characteristic matrix} $\Delta(\mu,z)\in\mathbb{C}^{nk\times nk}$
  for the problem \eqref{eq:ivpext}-\eqref{eq:icext} as
\begin{equation}\label{eq:deltaedef}
  \Delta(\mu,z)v := \Delta(\mu,z)
  \begin{bmatrix}
    v_0\\ \vdots\\ v_{k-1}
  \end{bmatrix} :=
  \begin{bmatrix}
    v_{0}- y(0)\\
    v_1-\lim_{t\nearrow t_1} y(t)\\
    \vdots\\
    v_{k-1}-\lim_{t\nearrow t_{k-1}} y(t)
  \end{bmatrix}\mbox{.}
\end{equation}
and the corresponding \textbf{characteristic function} $h(\mu,z)$ as
\begin{equation}
  \label{eq:hzdef}
  h(\mu,z) := \det\Delta(\mu,z)\mbox{.}
\end{equation}

\end{definition}

Note that the integral equation \eqref{eq:varpext} implies that
$v_j=\lim_{t\searrow t_j} y(t)$. In this sense the values $v_j$ are
the initial (or restart) values for the differential equation
\eqref{eq:ivpext}. The construction of $\Delta$ gives a well-defined
matrix for all $z$ and $\mu$ satisfying $\operatorname{Re}\mu\geq-R$ and
$|z|\leq\exp(R)$: for a given tuple $v$ we evaluate the unique
solution $y=[ I -zL_k(\mu)]^{-1}S(\mu)v\in C_k$ of the integral
equation \eqref{eq:varpext} and then we use this $y$ to evaluate the
right-hand side of \eqref{eq:deltaedef}.  The characteristic matrix
$\Delta$ is set up such that for $v$ in the kernel of $\Delta(\mu,z)$
the solution $y$ also satisfies the differential equation
\eqref{eq:ivpext} on the whole interval (and does not have jumps)
including the periodic boundary conditions:
\begin{lemma}[Differentiability]
  \label{thm:diff}
  Let $\operatorname{Re}\mu\geq-R$, $|z|\leq\exp(R)$ and $k>C(R)$. If the tuple
  $v=(v_0,\ldots,v_{k-1})$ satisfies
  \begin{equation}
    \label{eq:charmateq}
    \Delta(\mu,z)v=0
  \end{equation} 
  then $y=[ I -zL_k(\mu)]^{-1}S(\mu)v$, the solution of the integral
  equation \eqref{eq:varpext}, is continuously differentiable on
  $[-1,0]$ and satisfies the differential equation \eqref{eq:zbvp}
  with the periodic boundary condition \eqref{eq:zbc}. Conversely, let
  $y \in C^1([-1,0], \mathbb{C}^n)$ be a continuously differentiable
  solution of \eqref{eq:zbvp}--\eqref{eq:zbc}. Then the tuple
  $v=(v_0,\ldots,v_{k-1}) \in \mathbb{C}^k$, where $v_{0} = y(-1)$, $v_1 =
  y(t_1)$, $\dots$, $v_{k-1} = y(t_{k-1})$ satisfies
  \eqref{eq:charmateq}.
\end{lemma}
\begin{proof}
  The rows $2$ to $k$ of the right-hand side in the definition
  \eqref{eq:deltaedef} of $\Delta$ ensure continuity of $y$: since
  $v_j=\lim_{t\searrow t_j} y(t)$ for $j=1,\ldots,k-1$ the condition
  that these rows are equal to zero reads $\lim_{t\searrow t_j}
  y(t)=\lim_{t\nearrow t_j} y(t)$, and (by inserting the right-hand
  side of \eqref{eq:varpext})
  \begin{displaymath}
    v_{j+1}=U(t_{j+1},t_j,\mu)v_i+
    z\int\limits_{t_j}^{t_{j+1}}U(t_{j+1},s,\mu)B(s)
    y((s-\tau)_{\mathrm{mod}[-1,0]}){\mathrm{d}} s
  \end{displaymath}
  for $j=0,\ldots,k-1$. This implies that $y$ is continuous on
  $[-1,0]$, and that we can concatenate all the integral terms in
  \eqref{eq:varpext} to
  \begin{equation}\label{eq:int01}
    y(t)=U(t,-1,\mu)v_0+z\int_{-1}^tU(t,s,\mu) B(s)
    y((s-\tau)_{\mathrm{mod}[-1,0]}){\mathrm{d}} s
  \end{equation}
  The first row of the condition $\Delta(\mu,z)v=0$ reads $v_0= y(0)$,
  which makes sure that $y$ is periodic at the boundary of $[-1,0]$,
  and guarantees that the integrand in \eqref{eq:int01} is continuous
  at $s=\tau-1$. Consequently, the integrand is continuous everywhere,
  which implies that $y$ is continuously differentiable. Thus, we can
  differentiate \eqref{eq:int01} with respect to $t$, which implies
  that $y$ satisfies the differential equation \eqref{eq:zbvp}. This,
  in turn, implies that also $\dot y(-1)=\dot y(0)$ because the
  right-hand side of \eqref{eq:zbvp} is periodic.
\end{proof}

Since the map $L_k$ depends analytically on $\mu$, the matrix $\Delta$
depends analytically on $\mu$ and $z$, too (as long as
$\operatorname{Re}\mu\geq-R$ and $|z|\leq \exp(R)$). Calling $\Delta$ the
characteristic matrix makes sense for the following reason:
\begin{lemma}[Characteristic function for \eqref{eq:lindde} \cite{SS11}]
  \label{thm:char} Let $N\geq1$ and $R>0$ be given, and choose the
  number of sub-intervals, $k$, greater than $C(R)$. Then for all
  $\mu$ satisfying
  \begin{displaymath}
    \operatorname{Re}\mu\geq-\frac{R}{N+\tau}
  \end{displaymath}
  the following equivalence holds: $\mu$ is a Floquet exponent of the
  time-$1$ map $M_N$ of the periodic linear DDE \eqref{eq:lindde} if
  and only if
  \begin{equation}\label{eq:hndef}
    h_N(\mu):=h(\mu,\exp(-(N+\tau)\mu)) 
    = \det\Delta(\mu,\exp(-(N+\tau)\mu))=0\mbox{.}
  \end{equation}
  The algebraic multiplicity of $\mu$ as a Floquet exponent of $M_N$ is
  equal to the order of $\mu$ as a root of $h_N$.
\end{lemma}
\begin{proof}
  That $\mu$ is a Floquet exponent of \eqref{eq:lindde} if and only if
  $\Delta(\mu,\exp(-(N+\tau)\mu))$ has a non-trivial kernel is clear
  because Floquet exponents of $M_N$ were defined as those complex
  numbers for which \eqref{eq:zbvp}--\eqref{eq:zbc} has a non-trivial
  solution. The statement about the multiplicity of $\mu$ follows from
  arguments similar to \cite{KL92}, which are laid out in detail in
  \cite{SS11} (an equivalence between the eigenvalue problem for
  $M_N$, $\exp(\mu)v-M_Nv=0$, and the algebraic equation
  $\Delta(\mu,\exp(-(N+\tau)\mu))v=0$ is constructed in \cite{SS11}).
\end{proof}
Both functions, $h$ and $h_N$, are real analytic (that is, their
expansion coefficients are real numbers since $A$, $B$ and $\tau$ are
real). 

For any given $R > 0$ we have chosen $k > 0$ and constructed a
characteristic function $h_N(\mu)$, defined on the half plane 
\begin{displaymath}
  \left
    \{ \mu \in \mathbb{C} \mid \operatorname{Re}\mu\geq-\frac{R}{N+\tau} \right \}\mbox{,}
\end{displaymath}
the roots of which are precisely the Floquet exponents (counting
multiplicities) of the period map $M_N$ for the linearized delay
differential equation \eqref{eq:lindde}.

As we are interested in the location of Floquet exponents of the
time-$1$ map $M_N$ to the right or close to the imaginary axis we have
reduced the eigenvalue problem to a study of the asymptotic behavior
of roots of the holomorphic function $h_N$ for $N\to\infty$.
Furthermore, the asymptotic spectra are defined as roots and root
curves of $h$ by construction of $h$: $\mu>0$ is in the strongly
unstable spectrum $\mathcal{A}_+$ if and only if $h(\mu,0)=0$, and
$\mu=\gamma+i\omega$ is in the asymptotic continuous spectrum
$\mathcal{A}_c$ if and only if $h(i\omega,\exp(-\gamma-i\varphi))=0$
for some phase $\varphi\in\mathbb{R}$.

This proves Lemma~\ref{thm:charinformal} and allows us to follow an
approach similar to \cite{YW10,LWY11} in the following sections. 
\new{Note that in the case where the period and the delay are rationally dependent, the 
existence of a suitable characteristic function follows much easier (see \cite{SW06}),
but the structure of the asymptotic spectra discussed below will be not affected. }
\section{Strongly unstable spectrum}
\label{sec:strong}
One immediate consequence of Lemma~\ref{thm:charinformal} is the
statement asserted in \cite{YP09} for the relation between eigenvalues
of $M_N$ and the strongly unstable spectrum $\mathcal{A}_+$. Note that
$\mu$ is in the strongly unstable spectrum $\mathcal{A}_+$ if and only if
$h(\mu,0)=0$: if one sets $z=0$ in \eqref{eq:zbvp} the differential
equation reduces to \eqref{eq:str:ode}, the expression defining
$\mathcal{A}_+$. For large $N$ unstable Floquet exponents of $M_N$ either
approach the imaginary axis or $\mathcal{A}_+$:
\begin{lemma}[Convergence to strongly unstable spectrum]\label{thm:sus}\ \\
  Let $\operatorname{Re}\mu_0>0$.  If $h(\mu_0,0)\neq 0$ then there exists a $N_0$
  such that $\mu_0$ lies in the resolvent set of the time-$1$ map $M_N$
  for all $N>N_0$.

  If $\mu_0$ is a root of multiplicity $k$ of the function
  $h(\cdot,0)$ then every sufficiently small neighborhood $U$ of
  $\mu_0$ contains exactly $k$ Floquet exponents (counting
  multiplicity) of $M_N$ for all $N>N_0$ ($N_0$ depends on $U$).
\end{lemma}
The resolvent set is the set of all complex numbers $\mu$ for which the
map $\exp(\mu) I -M_N$ is an isomorphism (characterized by
$h_N(\mu)\neq0$ for $\operatorname{Re}\mu\geq -R/(N+\tau)$).

\begin{proof}
  The statement for $h(\mu_0,0)\neq0$ follows from expansion of
  $h(\mu_0,\cdot)$ in $z=0$ since $z=\exp(-(N+\tau)\mu_0)\to0$ for
  $N\to\infty$ if $\operatorname{Re}\mu_0>0$.

  Let $\mu_0$ be a root of $h(\cdot,0)$ of multiplicity $k$. For any
  sufficiently small $\delta>0$, $\mu_0$ is the only root of
  $h(\cdot,0)$ inside the ball $B_\delta(\mu_0)$ of radius $\delta$
  around $\mu_0$.  In particular, $h(\mu,0)\neq0$ on the boundary of
  $B_\delta(\mu_0)$.  Then $h(\mu,\exp(-(N+\tau)\mu))$ is also
  non-zero on the boundary of $B_\delta(\mu_0)$ for sufficiently large
  $N$, and the logarithmic derivative of $h(\mu,\exp(-(N+\tau)\mu))$
  converges for $N\to\infty$:
  \begin{align*}
    \lefteqn{\frac{\frac{{\mathrm{d}}}{{\mathrm{d}}\mu}[h(\mu,\exp(-(N+\tau)\mu))]}{
        h(\mu,\exp(-(N+\tau)\mu))}
      =}\\
    &\frac{\partial_1h(\mu,\exp(-(N+\tau)\mu))-(N+\tau)\exp(-(N+\tau)\mu)
      \partial_2h(\mu,\exp(-(N+\tau)\mu))}{h(\mu,\exp(-(N+\tau)\mu))}\\
    &\to_{N\to\infty} \frac{\partial_1h(\mu,0)}{h(\mu,0)}\mbox{\quad
      if $\operatorname{Re} \mu>0$ for all $\mu\in B_\delta(\mu_0)$.}
  \end{align*}
  Thus, the line integral of the logarithmic derivative of
  $h(\mu,\exp(-(N+\tau)\mu))$ along the boundary of $B_\delta(\mu_0)$,
  which counts the roots (and poles), also converges for $N\to\infty$
  to the logarithmic derivative of $h(\cdot,0)$. Since the line
  integral of the logarithmic derivative is an integer it must be
  constant for $N\to\infty$, and, hence, $h_N(\mu)$ must have the same
  number of roots (counting multiplicity) inside $B_\delta(\mu_0)$ as
  $h(\cdot,0)$, which is $k$.
\end{proof}

Lemma~\ref{thm:sus} shows that $M_N$ is exponentially unstable for all
sufficiently large $N$ if the strongly unstable spectrum is non-empty,
which is condition~\ref{thm:strongunst} of Theorem~\ref{thm:main}.

\section{Asymptotic continuous spectrum}
\label{sec:pc}
Due to Lemma~\ref{thm:charinformal} the asymptotic continuous spectrum
is given as those $\mu=\gamma+i\omega$ for which one can find a phase
$\varphi\in\mathbb{R}$ such that $h(i\omega,\exp(-\gamma-i\varphi))=0$.

Thus, we expect the asymptotic continuous spectrum to come in curves:
Let $\gamma_0 + i \omega_0 \in \mathcal{A}_c$ be a point in the
asymptotic continuous spectrum and let $\varphi_0\in[-\pi,\pi)$ be its
phase. By definition $\mu_0=\gamma_0+i\varphi_0 \in \mathbb{C}$
satisfies $0=h(i\omega_0,\exp(-\mu_0))$. If
$\partial_2h(i\omega_0,\exp(-\mu_0))\neq0$ then there is a root curve
$\mu(\omega)$ of complex numbers satisfying
$h(i\omega,\exp(-\mu(\omega)))=0$ going through $\mu_0$. Hence locally
(for $\omega$ in a neighborhood of $\omega_0$) there is a curve
$\omega\mapsto (\operatorname{Re} \mu(\omega)+i \omega) \in
\mathcal{A}_c$ through $\gamma_0 + i \omega_0$ for $\omega$ near
$\omega_0$ .

We assume for the remainder of the section that the instantaneous
spectrum $\Sigma_A$ has a positive distance to the imaginary axis. The
idea behind the construction of the asymptotic continuous spectrum is
that for Floquet exponents $\mu$ close to the imaginary axis, that is,
for $\mu$ of the form
\begin{equation}\label{eq:muscal}
  \mu=\frac{\gamma}{N+\tau}+i\omega
\end{equation}
with a bounded factor $\gamma$ in the real part, the roots of the
characteristic function $h_N(\mu)=h(\mu,\exp(-(N+\tau)\mu))$ converge to a
regular limit for $N\to\infty$ after inserting the scaling
\eqref{eq:muscal}. This can be made more specific: if the
instantaneous spectrum $\Sigma_A$ has a positive distance from the
imaginary axis then all Floquet exponents that are not converging to
the strongly unstable spectrum have a real part less than
$R_2/(N+\tau)$ where $R_2$ does not depend on $N$:
\begin{lemma}[Convergence to the imaginary axis]\label{thm:imag}
  Assume that the elements of the instantaneous spectrum $\Sigma_A$
  have a positive distance $2\epsilon>0$ from the imaginary axis. Let
  us denote the points in the strongly unstable spectrum, $\mathcal{
    A}_+$, by $\mu_1$,\ldots, $\mu_j$. There exists a constant
  $R_2\geq0$ such that all Floquet exponents $\mu$ of $M_N$ satisfy
  either $\mu\in B_\epsilon(\mu_j)$ for some $j$, or
  \begin{equation}\label{eq:muupbound}
    \operatorname{Re}\mu<\frac{R_2}{N+\tau}
  \end{equation}
for all $N\geq 1$.
\end{lemma}
%
\begin{proof}
  For all complex numbers $\mu$ that have non-negative real part and
  are outside of the balls $B_\epsilon(\mu_j)$ the matrix $ I
  -U(0,-1,\mu)$ is invertible and the inverse has a uniform upper
  bound (remember that $U$ is the propagation matrix of $\dot
  y=[A(t)-\mu I ]y$):
  \begin{displaymath}
    \left\|[ I -U(0,-1,\mu)]^{-1}\right\|\leq C\mbox{.}
  \end{displaymath}
  We also know that $\mu$ is a Floquet exponent if the integral
  equation \eqref{eq:int01} (which is equivalent to the differential
  equation \eqref{eq:charmateq}) has a non-trivial periodic solution
  $y(t)$ for $z=\exp(-(N+\tau)\mu)$. This implies ($y(-1)=y(0)=v_0$)
  \begin{displaymath}
    v_0=U(0,-1,\mu)v_0+z\int_{-1}^0U(0,s,\mu) B(s)
    y((s-\tau)_{\mathrm{mod}[-1,0]}){\mathrm{d}} s\mbox{,}
  \end{displaymath}
  and, hence,
  \begin{equation}\label{eq:yfixpoint}
    \begin{split}
      y(t)=&z\left[[ I -U(0,-1,\mu)]^{-1}\int_{-1}^0U(0,s,\mu) B(s)
        y((s-\tau)_{\mathrm{mod}[-1,0]}){\mathrm{d}} s+\right.\\
      &\phantom{z}+\left.\int_{-1}^tU(t,s,\mu) B(s)
        y((s-\tau)_{\mathrm{mod}[-1,0]}){\mathrm{d}} s\right]
    \end{split}
  \end{equation}
  The factor of $z$ in the right-hand side of this fixed-point problem
  for $y$ is a linear operator $K(\mu)$ on the space of continuous
  functions. $K(\mu)$ is uniformly bounded for all $\mu$ that have
  non-negative real part and are outside of the balls
  $B_\epsilon(\mu_j)$. Let us denote an upper bound on the norm of
  $K(\mu)$ by $C_1$. Then for all $z$ satisfying $|z|<C_1^{-1}$ the
  integral equation \eqref{eq:int01} cannot have a non-trivial
  solution. Since, for Floquet exponents of $M_N$, $z$ equals
  $\exp(-(N+\tau)\mu)$ this means that $\mu$ has to satisfy
  \begin{displaymath}
    |\exp(-(N+\tau)\mu)|\geq C_1^{-1}\mbox{,}
  \end{displaymath}
  and, thus
  \begin{displaymath}
    \operatorname{Re}\mu\leq\frac{\log C_1}{N+\tau}
  \end{displaymath}
  for $|z|$ to be larger than $C_1^{-1}$. Consequently, if we choose
  $R_2>\log C_1$ then a $\mu$ that has non-negative real part and lies
  outside of the balls $B_\epsilon(\mu_j)$ must satisfy
  \eqref{eq:muupbound} in order to be a Floquet exponent of
  $M_N$.
\end{proof}

A corollary of the construction described by
equation~\eqref{eq:yfixpoint} is that for each $\omega$ the
intersection of the asymptotic continuous spectrum $\mathcal{A}_c$
with the horizontal line
$\{\mu\in\mathbb{C}:\operatorname{Im}\mu=\omega\}$ can be described as
a set of eigenvalues of a compact operator $K(\i\omega)$. Thus, if
$\operatorname{Re}\Sigma_A\neq0$ the asymptotic continuous spectrum
forms curves parametrizable by $\omega$, which may only have poles at
$-\infty$:
\begin{corollary}[Asymptotic continuous spectrum consists of
  curves]\label{thm:pcbounded}
  Assume that the instantaneous spectrum $\Sigma_A$ has a positive
  distance to the imaginary axis.  A point $\mu$ is in the asymptotic
  continuous spectrum if $\exp(\mu)$ is an eigenvalue of the compact
  linear operator 
  $K(\i\omega):C([-1,0];\mathbb{C}^n)\mapsto C([-1,0];\mathbb{C}^n)$
  given by
  \begin{align*}
    K(\i\omega)y:=&\left[[ I
      -U(0,-1,\i\omega)]^{-1}\int_{-1}^0U(0,s,\i\omega) B(s)
      y((s-\tau)_{\mathrm{mod}[-1,0]}){\mathrm{d}} s+\right.\\
    &\phantom{z}+\left.\int_{-1}^tU(t,s,\i\omega) B(s)
      y((s-\tau)_{\mathrm{mod}[-1,0]}){\mathrm{d}} s\right]\mbox{.}
  \end{align*}
  In particular there exists a constant $R_3\geq0$ so that the
  asymptotic continuous spectrum lies to the left of the vertical line
  $\{z\in\mathbb{C}:\operatorname{Re} z=R_3\}$.

\end{corollary}
The matrix $U$ in the definition of $K$ is the monodromy matrix of the
instantaneous problem $\dot x(t)=[A(t)-\i\omega]x(t)$.  The inverse of
$ I -U(0,-1,\i\omega)$ exists because of our assumption that the
instantaneous spectrum, $\Sigma_A$, does not contain points on the
imaginary axis. The corollary implies that the asymptotic continuous
spectrum consists of curves even in points where the regularity
condition on $\partial_2h\neq0$ is violated. In these points
$K(\i\omega)$ has eigenvalues of finite multiplicity larger than $1$
such that finitely many curves of the asymptotic continuous spectrum
cross each other (or are on top of each other).

All Floquet exponents of $M_N$ that do not converge to the strongly
unstable spectrum for $N\to\infty$ (these can be at most $n$,
equaling the dimension of $A(t)$) are either stable, 
or they satisfy restriction \eqref{eq:muupbound} on the upper
bound. Thus, apart from the strongly unstable spectrum, the
multipliers, which are of interest from the point of view of stability
and bifurcations, lie in the strip
\begin{displaymath}
  \mathcal{C}_N:=\left\{\mu\in\mathbb{C}:-\frac{R}{N+\tau}\leq 
    \operatorname{Re}\mu\leq \frac{R_2}{N+\tau}\right\}
\end{displaymath}
and have the form
\begin{equation}\label{eq:floqscal}
  \mu=\frac{\gamma}{N+\tau}+\i\omega
  \mbox{\quad where $\gamma\in[-R,R_2]$ and $\omega\in[-\pi,\pi)$}
\end{equation}
(after shifting them into the strip $\{z: \operatorname{Im} z\in[-\pi,\pi)\}$ by
subtraction of an integer multiple of $2\pi \i$). As we focus our
discussion on Floquet exponents of the scale \eqref{eq:floqscal} from
now on it makes sense to introduce the notation of a \emph{scaled
  Floquet exponent}.
\begin{definition}[Scaled Floquet exponent \& resolvent set]\ \\ A complex
  number $\mu=\gamma+\i\omega$ with $\gamma\in[-R,R_2]$ and
  $\omega\in[-\pi,\pi)$ is called a \textbf{scaled Floquet exponent}
  of $M_N$ if $\gamma/(N+\tau)+\i\omega$ is a Floquet exponent of
  $M_N$. Similarly, $\mu$ is in the \textbf{scaled resolvent set} of
  $M_N$ if $\gamma/(N+\tau)+\i\omega$ is in the resolvent set of $M_N$.
\end{definition}
We can now formulate precisely in which sense the asymptotic
continuous spectrum $\mathcal{A}_c$ is the limit of the spectra of
$M_N$. First, we make a statement about resolvent sets.


\begin{lemma}[Points distant from $\mathcal{A}_c$]\label{thm:pcdist}
  Let the instantaneous spectrum $\Sigma_A$ have a positive distance
  to the imaginary axis. If $\gamma_0>-R$ and $\gamma_0 + i \omega_0$
  is not in the asymptotic continuous spectrum $\mathcal{A}_c$ then
  $\gamma_0+i\omega_0$ is in the scaled resolvent set of $M_N$ for
  sufficiently large $N$.
\end{lemma}

\begin{proof}
  Lemma~\ref{thm:pcdist} is a simple consequence of the fact that, if
  $h(i\omega_0,\exp(-\gamma_0-i\varphi))\neq0$ for all phases
  $\varphi\in[-\pi,\pi]$ (which defines points
  $\gamma_0+i\omega_0\notin\mathcal{A}_c$ if $\gamma_0>-R$) then
  \begin{displaymath}
    h_N\left(\frac{\gamma_0}{N+\tau}+i\omega_0\right)=
    h\left(i\omega_0+\frac{\gamma_0}{N+\tau},
      \exp\left(-\gamma_0-i(N+\tau)\omega_0 \right)\right)\neq0  
  \end{displaymath}
  for all sufficiently large $N$. Lemma~\ref{thm:charinformal} implies
  that $\gamma_0/(N+\tau)+i\omega_0$ is in the resolvent set of $M_N$
  for all sufficiently large $N$.
\end{proof}

Lemma~\ref{thm:pcdist} implies that all spectrum of $M_N$ near the
imaginary axis has to be close to $\mathcal{A}_c$ even after
\new{blow-up of the real part}. The other direction, that every point
of $\mathcal{A}_c$ in the strip $-R<\operatorname{Re}\mu<R_2$ is
approached by scaled Floquet exponents of $M_N$, is also true (again
under the assumption that the instantaneous spectrum $\Sigma_A$ is not
on the imaginary axis). We prove this direction first for regular root
curves of $h$ because this is the most common case, and the missing
piece for the stability criterion. Our estimate is slightly sharper
than mere approximation to make it useful for our proof of the
stability criterion.

In short, Lemma~\ref{thm:pccurves} below states that any regular curve
of asymptotic continuous spectrum, $\tilde\gamma(\omega)+\i\omega$ (with its
corresponding phase $\tilde\varphi(\omega)$), is approximated by scaled Floquet
exponents of the form $\gamma_k+\i\omega_k$ where
\begin{align}
  \label{eq:omega:inf}
  \omega_k&=\frac{2k\pi}{N+\tau}+\frac{1}{N+\tau}
  \tilde\varphi\left(\frac{2k\pi}{N+\tau}\right)+O\left((N+\tau)^{-2}\right)\\
  \label{eq:gamma:inf}
  \gamma_k&=\left[1+O\left((N+\tau)^{-1}\right)\right]\tilde\gamma(\omega_k)
  \mbox{,}
\end{align}
$k$ are integers such that $2k\pi/(N+\tau)$ is near $\omega$, and the
$O$-s are smooth functions of $\omega$. This proves the stronger
estimate \eqref{eq:omega1}--\eqref{eq:gamma1}, given in the non-technical
overview in Section~\ref{sec:nontech}.

\begin{lemma}[Convergence to regular curves of $\mathcal{A}_c$]
\label{thm:pccurves}\ \\
  Assume that the triplet $(\omega_*,\gamma_*,\varphi_*)$ satisfies
  \begin{displaymath}
    h(\i\omega_*,\exp(-\gamma_*-\i\varphi_*))=0\mbox{,\quad}
    \partial_2h(\i\omega_*,\exp(-\gamma_*-\i\varphi_*))\neq0\mbox{.}
  \end{displaymath}
  Let $N_*$ be sufficiently large and $\delta>0$ be sufficiently
  small. We denote the unique regular root curve of $\mu\mapsto
  h(\i\omega,\exp(-\mu))$ for $\omega$ near $\omega_*$ through
  $(\omega=\omega_*,\mu_*=\gamma_*+\i\varphi_*)$ by
  \begin{displaymath}
    \tilde\mu(\omega)=\tilde\gamma(\omega)+\i\tilde\varphi(\omega)\mbox{.}  
  \end{displaymath}
  For all $N\geq N_*$ there are scaled Floquet exponents $\mu_{N,k}$
  of $M_N$ near $\gamma_*+i\omega_*$. \new{All scaled Floquet exponents
  $\mu_{N,k}$ that are sufficiently close to $\gamma_*+i\omega_*$} are
  algebraically simple and have the form
  \begin{equation}\label{eq:appfloq}
    \mu_{N,k}=\tilde\gamma_N(\omega_k)+i\omega_k\mbox{,}
  \end{equation}
  where:
  \begin{itemize}
  \item $k$ is any integer satisfying
    \begin{displaymath}
      \frac{2k\pi}{N+\tau}\in(\omega_*-\delta,\omega_*+\delta)\mbox{,}
    \end{displaymath}
  \item $\omega_k$ is the unique solution of the fixed point problem
    for $\omega$
    \begin{equation}
      \label{eq:pc:imag:ex}
      \omega=\frac{\tilde\varphi_N(\omega)}{N+\tau}+\frac{2k\pi}{N+\tau}\mbox{,}
    \end{equation}
  \item the functions $\tilde\gamma_N$ and $\tilde\varphi_N$ are perturbations of
    $\tilde\gamma(\omega)$ and $\tilde\varphi(\omega)$ of the form
    \begin{align}
      \tilde\gamma_N(\omega)&= \left[1+ \frac{1}{N+\tau}\operatorname{Re}
        g\left(\omega,\frac{1}{N+\tau}\right)\right]
      \tilde\gamma(\omega)\label{eq:gammaN}\\
      \tilde\varphi_N(\omega)&= \tilde\varphi(\omega)+ \frac{1}{N+\tau}\operatorname{Im}
      g\left(\omega,\frac{1}{N+\tau}\right)\tilde\gamma(\omega)\label{eq:phiN}
    \end{align}
    where $g(\omega,\epsilon)$ is a smooth complex-valued function,
    which is independent of $N$ and defined for
    $\omega\in(\omega_*-\delta,\omega_*+\delta)$ and
    $\epsilon\in[0,1/N_*)$.
  \end{itemize}
\end{lemma}
We notice that the scaled Floquet exponents of $M_N$ lie on bands:
they are on the curve given by $\tilde\gamma_N(\omega)+\i\omega$, which
is a small perturbation of the curve of asymptotic continuous spectrum
$\tilde\gamma(\omega)+\i\omega$. The spacing between scaled Floquet
exponents along this band is given by the fixed point equation
\eqref{eq:pc:imag:ex}. The fixed point problem \eqref{eq:pc:imag:ex}
is only weakly implicit since the right-hand side terms containing
$\omega$ all have a pre-factor $1/(N+\tau)$, which is small. Hence,
expression~\eqref{eq:gamma:inf} for $\gamma_k$ follows immediately
from \eqref{eq:gammaN} and expression~\eqref{eq:omega:inf} for
$\omega_k$ follows from \eqref{eq:pc:imag:ex} and \eqref{eq:phiN}.

\begin{proof}
  (of Lemma~\ref{thm:pccurves})
  We know that $\tilde\gamma(\omega)+\i\tilde\varphi(\omega)$ is a
  regular root curve of
  \begin{equation}\label{eq:pc:rootcurve}
    h(i\omega,\exp(-\gamma-\i\varphi))=0\mbox{.}
  \end{equation}
  The root problem
  \begin{displaymath}
    h\left(\i\omega+\epsilon\gamma,
      \exp\left(-\gamma-\i\varphi\right)\right)=0
  \end{displaymath}
  is for small $\epsilon$ a small (order $\epsilon$) perturbation of
  the root problem \eqref{eq:pc:rootcurve}. Thus, for sufficiently
  small $\epsilon$, a root curve of the form
  $\tilde\gamma_\epsilon(\omega)+\i\tilde\varphi_\epsilon(\omega)$
  exists for $\omega$ in some neighborhood of $\omega_*$, and it has
  the form
  \begin{equation}\label{eq:pc:expand}
    \tilde\gamma_\epsilon(\omega)+\i\tilde\varphi_\epsilon(\omega)= 
    \left[1+\epsilon g\left(\omega,\epsilon\right)\right]
    \tilde\gamma(\omega)+\i\tilde\varphi(\omega)\mbox{,}
  \end{equation}
  where $g(\omega,\epsilon)$ is a smooth complex-valued function
  defined for $\omega$ in a small neighborhood of $\omega_*$ and
  $\epsilon\in[0, \epsilon_{\max})$ (with some
  $\epsilon_{\max}>0$). Note that the error term contains a factor
  $\tilde\gamma(\omega)$, making the error equal to zero on the
  imaginary axis. (See Appendix~\ref{app:sec:mult} for details of how
  to extract this factor from the error.) Inserting
  $\epsilon=(N+\tau)^{-1}$ and labeling the curves
  $\tilde\gamma_\epsilon(\omega)$ as $\tilde\gamma_N(\omega)$ and
  $\tilde\varphi_\epsilon$ as $\tilde\varphi_N$ gives the definitions
  \eqref{eq:gammaN} and \eqref{eq:phiN} in the lemma. Correspondingly,
  we make an initial choice for the minimal $N$, $N_*$, as
  $1/\epsilon_{\max}$.

  A point on the curve $\tilde\gamma_N(\omega)+\i\omega$ is a
  \new{scaled} Floquet exponent of $M_N$ if and only if its imaginary
  part $\omega$ satisfies
  \begin{align*}
    \exp(-\i\omega(N+\tau))&=\exp(-\i\tilde\varphi_N(\omega))\mbox{\quad
      and,
      thus,}\\
    (N+\tau)\omega&=\tilde\varphi_N(\omega)+2k\pi\mbox{\quad\ \ for
      some integer $k \in \mathbb{Z}$.}
  \end{align*}
  After dividing by $N+\tau$ this becomes the fixed point equation
  \eqref{eq:pc:imag:ex} for $\omega$. For which $k$ does this fixed
  point problem have a unique solution?

  We choose a neighborhood $\mathcal{U}$ of $\omega_*$ of the form
  $(\omega_*-2\delta,\omega_*+2\delta)$ such that $\tilde\varphi_N$ is
  well-defined for all $\omega\in\mathcal{U}$ and all $N\geq N_*$, and
  satisfies $|\tilde\varphi_N'(\omega)|<L$ for some constant $L$
  (which is independent of $N$). Next, we increase $N_*$ such that
  \begin{align*}
    \frac{L}{N+\tau}&<1 \mbox{\quad for all $N>N_*$, and}\\
    \frac{\tilde\varphi_N(\omega)}{N+\tau}&<\delta\mbox{\quad for all
      $N\geq N_*$ and $\omega\in\mathcal{U}$.}
  \end{align*}
  Then the right-hand side of the fixed point problem
  \eqref{eq:pc:imag:ex} is contracting with a Lipschitz constant
  $L/(N+\tau)$ for all $\omega\in\mathcal{U}$ and it is mapping
  $\mathcal{ U}=(\omega_*-2\delta,\omega_*+2\delta)$ back into itself:
  \begin{displaymath}
    \left|\frac{\tilde\varphi_N(\omega)}{N+\tau}+\frac{2k\pi}{N+\tau}-\omega_*\right|
    \leq\left|\frac{\tilde\varphi_N(\omega)}{N+\tau}\right|
    +\left|\frac{2k\pi}{N+\tau}-\omega_*\right|<
    2\delta
  \end{displaymath}
  for all $\omega\in\mathcal{U}$ and all $N\geq N_*$ if
  \begin{equation}\label{eq:krestrict}
    \left|\frac{2k\pi}{N+\tau}-\omega_*\right|<\delta\mbox{.}
  \end{equation}
  Thus, the Banach Contraction Mapping Principle guarantees that
  \eqref{eq:pc:imag:ex} has a unique solution $\omega_k$ for all
  integers $k$ satisfying \eqref{eq:krestrict}.

  Finally, we confirm the algebraic simplicity of the scaled Floquet
  exponents $\mu_{N,k}=\tilde\gamma_N(\omega_k)+\i\omega_k$ by
  checking the multiplicity of the corresponding root of $h_N$: the
  derivative of $h_N$, divided by $N+\tau$, is
  \begin{displaymath}
    \frac{h_N'(z)}{N+\tau}=
    \frac{\partial_1h(z,\exp(-(N+\tau)z))}{N+\tau}-
    \exp(-(N+\tau)z)\partial_2h(z,\exp(-(N+\tau)z))\mbox{.}  
  \end{displaymath}
  Inserting $\tilde\gamma_N(\omega_k)/(N+\tau)+\i\omega_k$ for $z$ on
  the right, and the relation
  $\exp(-\i(N+\tau)\omega_k)=\exp(-\i\tilde\varphi_N(\omega_k))$ we
  get
  \begin{align*}
    \lefteqn{\frac{1}{N+\tau}h_N'(z)=O(N^{-1})-}\\
    -& \exp(-\tilde\gamma_N(\omega_k)-i\tilde\varphi_N(\omega_k))
    \partial_2h\left(i\omega_k+\frac{\tilde\gamma_N(\omega_k)}{N+\tau},
      \exp(-\tilde\gamma_N(\omega_k)-i\tilde\varphi_N(\omega_k))\right)\\
    =&O(N^{-1})-\exp(-\tilde\gamma_N(\omega_k)-i\tilde\varphi_N(\omega_k))
    \partial_2h(i\omega_k,\exp(-\tilde\gamma_N(\omega_k)-i\tilde\varphi_N(\omega_k)))
    \mbox{.}
  \end{align*}
  Since $\omega_k$ is in the neighborhood $\mathcal{U}$ of $\omega_*$
  in which $\partial_2h$ is non-zero the overall derivative is
  non-zero for sufficiently large $N$.
\end{proof}


Lemma~\ref{thm:pccurves} is based on a perturbation argument assuming
that the complex function $h(\i\omega,\exp(-\gamma-\i\varphi))$ has a
regular root curve
$\omega\mapsto\tilde\gamma(\omega)+\i\tilde\varphi(\omega)$. Thus, it
is also valid if the instantaneous spectrum, $\Sigma_A$, does not have
a positive distance to the imaginary axis as long as one restricts
consideration to Floquet exponents of the form
$\gamma/(N+\tau)+\i\omega$ (with bounded $\gamma$). Positive distance
of $\Sigma_A$ to the imaginary axis merely ensures that all Floquet
exponents $\mu$ with $\operatorname{Re}\mu>-R/(N+\tau)$ are of this
form (except for those approximating the strongly unstable spectrum
$\mathcal{A}_+$).

Lemma~\ref{thm:pcdist} and Lemma~\ref{thm:pccurves} about the
approximation of the asymptotic continuous spectrum, together with
Lemma~\ref{thm:sus} about the approximation of the strongly unstable
spectrum, are the tools that we need to prove the criterion for
asymptotic stability from Theorem~\ref{thm:main}. Before turning to
asymptotic stability let us prove the remaining statement about
spectral approximation of the asymptotic continuous spectrum
$\mathcal{ A}_c$. Let $\gamma_*+\i\omega_*$ be an element of
$\mathcal{A}_c$ with phase $\varphi_*$ (that is,
$h(\i\omega_*,\exp(-\gamma_*-\i\varphi_*))=0$). Then we find scaled
Floquet exponents of $M_N$ that approximate $\gamma_*+\i\omega_*$ even
if the non-degeneracy condition
$\partial_2h(\i\omega_*,\exp(-\gamma_*-\i\varphi_*))\neq0$ is not
satisfied:
\begin{lemma}[Approximation of $\mathcal{A}_c$]\label{thm:weakapp}
  Assume that the instantaneous spectrum $\Sigma_A$ is not on the
  imaginary axis. Let $\gamma_*\in[-R,R_2]$ and let
  $\gamma_*+\i\omega_*$ be in the asymptotic continuous spectrum $\mathcal{
    A}_c$. Then there exists a sequence $\gamma_N+\i\omega_N$ of
  scaled Floquet exponents of $M_N$ such that
  \begin{displaymath}
    \gamma_N+\i\omega_N\to\gamma_*+\i\omega_*\mbox{, \ as \ }N\to\infty \mbox{.}
  \end{displaymath}
\end{lemma}
Lemma~\ref{thm:weakapp} covers the claim of Theorem~\ref{thm:main}
about exponential instability of $M_N$ for sufficiently large $N$
under the condition of weak instability \ref{thm:acspos}.

\begin{proof}
  Since the instantaneous spectrum $\Sigma_A$ does not contain points
  on the imaginary axis, $\exp(\gamma_*+\i\varphi_*)$ is a non-zero
  eigenvalue of the compact operator $K(\i\omega_*)$ as introduced in
  Corollary~\ref{thm:pcbounded}. Consequently,
  $\exp(\gamma_*+\i\varphi_*)$ is isolated and has finite
  multiplicity, which implies that $\gamma_*+\i\varphi_*$ has finite
  multiplicity as a root of $z\mapsto h(\i\omega_*,\exp(-z))$. Let us
  define $k(N)$ for large $N$ as
  \begin{displaymath}
    k(N)=\mbox{\ greatest integer $k$ such that\ } 
    \frac{2k\pi}{N+\tau}\leq\omega_*\mbox{.}
  \end{displaymath}
  By construction of $k(N)$ we have that
  \begin{displaymath}
    \lim_{N\to\infty}\frac{2\pi k(N)}{N+\tau}=\omega_*\mbox{.}
  \end{displaymath}
  Since $z\mapsto h(\i\omega_*,\exp(-z))$ has an isolated root at
  $\gamma_*+\i\varphi_*$, the functions
  \begin{displaymath}
    z\mapsto h\left(\frac{2\pi \i k(N)}{N+\tau}+\frac{z}{N+\tau},
      \exp(-z)\right)
  \end{displaymath}
  also have roots $z_N=\gamma_N+\i\varphi_N$ for sufficiently large
  $N$ which converge to $\gamma_*+\i\varphi_*$ for $N\to\infty$. Let
  us define
  \begin{displaymath}
    \omega_N=\frac{2\pi k(N)+\varphi_N}{N+\tau}\mbox{.}
  \end{displaymath}
  Then, by construction, $\gamma_N+\i\omega_N$ is a scaled Floquet
  exponent of $M_N$ since
  \begin{displaymath}
    h\left(\frac{\gamma_N+2\pi \i k(N)+\i\varphi_N}{N+\tau},
      \exp(-\gamma_N-(2\pi k(N)+\varphi_N)\i)\right)=0\mbox{.}
  \end{displaymath}
  Moreover, $\gamma_N\to\gamma_*$ and $\omega_N\to\omega_*$, which
  proves the claim of the lemma.
\end{proof}

\section{Asymptotic stability for large delay}
\label{sec:stab}
The convergence results for the strongly unstable spectrum in
Lemma~\ref{thm:sus} and the asymptotic continuous spectrum in
Lemma~\ref{thm:pcdist} and Lemma~\ref{thm:pccurves} can be combined to
give a criterion for the stability of the periodic orbit $x_*$ of the
original nonlinear system \eqref{eq:nonlinddeN} depending on the
triplet $(A,B,\tau)$ that ensures stability of $x_*$ for all
sufficiently large $N$. If the instantaneous spectrum $\Sigma_A$ has a
positive distance to the imaginary axis then the point $0$ is part of
the asymptotic continuous spectrum because $0$ is a Floquet exponent
for all $N$. If $\partial_2h(0,1)\neq0$ then a regular curve
$\tilde\gamma(\omega)+\i\omega$ of the asymptotic continuous spectrum
is passing through $0$ (that is, $\tilde\gamma(0)=0$). This curve at
least touches the imaginary axis because $\gamma(0)=0$ and
$\tilde\gamma'(0)=0$ ($\tilde\gamma(\omega)$ is an even function,
thus, all odd derivatives of $\gamma$ are zero).

\begin{lemma}[Asymptotic stability]\label{thm:stab}
  Let the triplet $(A,B,\tau)$ be such that all elements of its
  instantaneous spectrum $\Sigma_A$ have negative real
  part. Furthermore, we assume that $\partial_2h(0,1)\neq0$, and that
  for the asymptotic continuous spectrum (including the
  corresponding phase $\varphi$)
 \begin{align*}
   \mathcal{A}_{c,\varphi}:=\bigl\{(\omega,\gamma,\varphi): & \
     h(\i\omega,\exp(-\gamma-\i\varphi))=0,  \\
     & \omega\in[-\pi,\pi)\mbox{,\ }\gamma\in[-R,R_2]\mbox{,\ }
     \varphi\in[-\pi,\pi]\bigr\}
 \end{align*}
 \new{the point} $(\gamma,\omega,\varphi)=(0,0,0)$ is the only point with
 $\gamma\geq0$. Then the map $M_N$ is orbitally exponentially stable
 for all sufficiently large $N$.
\end{lemma}
Note that the assumptions of Lemma ~\ref{thm:stab} exclude the case
$h(0,-1)=0$ since this would mean that $\gamma=\omega=0$, $\varphi=\pi$
is in $\mathcal{A}_{c,\varphi}$.

\begin{proof}
  (Lemma~\ref{thm:stab})
  Since $\partial_2h(0,1)\neq0$ we know that the Floquet exponent $0$
  is simple for $M_N$ if $N$ is sufficiently large. Also, since the
  instantaneous spectrum is in the negative half-plane the strongly
  unstable spectrum is empty. Hence, $M_N$ is the return map of a
  stable periodic orbit $x_*$ if it has no non-zero scaled Floquet
  exponent $\gamma+\i\omega$ for which $\gamma\in[0,R_2]$ (and
  $\omega\in[-\pi,\pi)$).

  Proving the statement by contradiction, we assume that, for a
  sequence of increasing $N$, $M_N$ has a scaled Floquet exponent
  $\gamma_N+\i\omega_N\neq0$ where $\gamma_N\in[0,R_2]$.

  The sequences $(\gamma_N,\omega_N,\varphi_N)$ where
  $\varphi_N=(N+\tau)\omega_N\mod[-\pi,\pi)$ must have accumulation
  points.  Without loss of generality we pick our sequence such that
  it converges to one of these accumulation points, say
  $(\gamma_*,\omega_*,\varphi_*)$.  Since
  $h_N(\gamma_N/(N+\tau)+i\omega_N)=0$ we have by definition of $h_N$
  and $\varphi_N$:
  \begin{align*}
    0&=h\left(\i\omega_N+\frac{\gamma_N}{N+\tau},\exp(-\gamma_N-\i\varphi_N)\right)
    \mbox{, and, thus, by continuity of $h$}\\
    0&=h(\i\omega_*,\exp(-\gamma_*-\i\varphi_*))\mbox{.}
  \end{align*}
  Consequently, the accumulation point must be an element of the
  asymptotic continuous spectrum $\mathcal{A}_{c,\varphi}$. Since
  $\gamma_N\geq0$ for all $N$ of the sequence, $\gamma_*$ must be
  greater or equal $0$, too. By assumption, the only element of
  $\mathcal{ A}_{c,\varphi}$ with non-negative $\gamma$ is
  $\gamma_*+\i\omega_*=0$. Thus, $\gamma_*=\varphi_*=\omega_*=0$.

  \new{We choose $\epsilon>0$ sufficiently small such that we can
    apply Lemma~\ref{thm:pccurves} to
    $(\gamma_*,\varphi_*,\omega_*)$.} We have that
  $\omega_N\in(-\epsilon,\epsilon)$ and $\gamma_N+\i\varphi_N\in
  B_\epsilon(0)$ for sufficiently large $N$ of the sequence
  $(\gamma_N,\omega_N,\varphi_N)$ (how large $N$ has to be, depends on
  $\epsilon$). Since $\partial_2h(0,1)\neq0$ this guarantees that
  $\gamma_N$ lies on the curve $\tilde\gamma_N(\omega)$ given in
  \eqref{eq:gammaN} in Lemma~\ref{thm:pccurves}:
  \begin{displaymath}
    \gamma_N=\tilde\gamma_N(\omega_N)=\left[1+\frac{1}{N+\tau} 
      g\left(\omega_N, \frac{1}{N+\tau}\right)\right]\tilde\gamma(\omega_N)
    \mbox{.}
  \end{displaymath}
  Since for all sufficiently large $N$ the factor in front of
  $\tilde\gamma(\omega_N)$ is positive, $\tilde\gamma_N(\omega_N)$
  must have the same sign as $\tilde\gamma(\omega_N)$, which is
  negative if $\omega_N\neq0$ due to the assumptions of the
  lemma. Since $\gamma_N$ is assumed to be non-negative, this implies
  that $\omega_N=\gamma_N=0$.  Thus, the scaled Floquet exponents
  $\gamma_N+\i\omega_N$ of $M_N$ are zero for the converging
  sub-sequence, which is in contradiction to our assumption
  $\gamma_N+\i\omega_N\neq0$.
\end{proof}


Lemma~\ref{thm:stab} proves the exponential stability claim of
Theorem~\ref{thm:main}.

\section{Conclusions}
We have shown that Floquet exponents of periodic solutions of delay
differential equations \eqref{eq:introdde} with large delay can be
approximated by a set of continuous curves (asymptotic continuous
spectrum) that are independent of the delay, and a finite set Floquet
exponents (strongly unstable spectrum).  Although the structure of the
spectrum is shown to be similar to the case of equilibria
\cite{LWY11}, there are some unique features, which occur specifically
for periodic orbits. Our results are based on the construction of the
characteristic function, the roots of which give Floquet multipliers
of the periodic orbit.

Using the asymptotic spectra we have been able to provide necessary
and sufficient conditions for the exponential stability of periodic
solutions for all sufficiently large delays.  Our results are
applicable to the case when the delay $\tau$ is large compared to the
period $T$ of the solution. In this case, the large parameter $N$,
which controls precision of the asymptotic approximation is
proportional to $\tau/T$.

Let us mention some of the specific features of the spectrum.  In
contrast to the equilibrium case, the asymptotic continuous spectrum
of Floquet exponents for periodic solutions contains generically a
curve with a tangency to the imaginary axis (see
Figure~\ref{fig:illu}(e,f)).  We have proved that even in the presence
of this tangency, the stability (or instability) of the asymptotic
continuous spectrum implies the exponential stability
(resp. instability) of the corresponding periodic orbit.  We have
shown that the generic decay rate of perturbations of the
exponentially stable periodic orbit of system \eqref{eq:introdde} is
of the order $N^{-3}$.

From the practical point of view, our results can be useful for
studying periodic regimes in applications that involve feedback with
large delays, for example, semiconductor lasers with optical
feedback \cite{LK80,YW10}, or systems with feedback control
\cite{SS08}.

From a mathematical point of view our result may provide a rigorous
approach to proving the existence of a large number of stable rapidly
oscillating periodic solutions for some special cases in which the
asymptotic spectra can be computed explicitly. This would provide a
contrast to the results for scalar feedback equations \cite{W81}.

\bibliographystyle{plain}

\appendix

\section{Proof of Lemma~\ref{thm:ivp}}
\label{sec:ivpproof}

We have to estimate the norm of $L_k$ with respect to the
$\|\cdot\|_\infty$ norm. The operator $L_k$ was defined as
\begin{displaymath}
  \left[L_k(\mu)y\right](t)=
  \int_{a_k(t)}^t U(t,s,\mu)B(s)y((s-\tau)_{\mathrm{mod}[-1,0]}){\mathrm{d}} s
\end{displaymath}
mapping a piecewise continuous function $y\in C_k$ back into $C_k$.
Using the norm estimate \eqref{eq:uest} we can estimate the norm of
$L_k$ by
\begin{align}
  \|L_k(\mu)y\|_\infty &\leq\max_{t\in[-1,0]}\left|\int_{a_k(t)}^t
    \|U(t,s,\mu)\|_\infty\|B\|_\infty\|y\|_\infty{\mathrm{d}} s\right|\nonumber\\
  &\leq\max_{t\in[-1,0]}\int_{a_k(t)}^t
    \exp([\|A\|_\infty-\operatorname{Re}\mu](t-s)){\mathrm{d}} s\cdot\|B\|_\infty\|y\|_\infty \mbox{.}
    \label{eq:varpextest}
\end{align}
We distinguish two sub-cases depending on the sign of $\|A\|_\infty-\operatorname{Re}\mu$:

\paragraph*{Case 1}
If $\|A\|_\infty\leq\operatorname{Re}\mu$ then the integrand in
\eqref{eq:varpextest} is bounded by unity such that
\begin{equation}\label{eq:lkest:case1}
  \|L_k(\mu)\|_\infty\leq (t-a_k(t))\|B\|_\infty\leq \frac{1}{k}\|B\|_\infty
  \mbox{,}
\end{equation}
taking into account that the length of the integration interval
$[a_k(t),t]$ in the right-hand side of \eqref{eq:varpext} is less than
$1/k$ for all $t\in[-1,0]$ by construction of $a_k(t)$ (see
\eqref{eq:varpexta}).  
\paragraph*{Case 2}
For the case $\operatorname{Re}\mu<\|A\|_\infty$ we can bound the
whole integral by (note that for any $a>0$, $t\geq0$ the
inequality $[\exp(at)-1]/a\leq at\exp(at)$ holds)
\begin{multline}\label{eq:app:case2}
  \frac{\exp\left(
      \left[\|A\|_\infty-\operatorname{Re}\mu\right](t-a_k(t))\right)-1}{
    \|A\|_\infty-\operatorname{Re}\mu}\leq\\
  (t-a_k(t))\exp\left(
    \left[\|A\|_\infty-\operatorname{Re}\mu\right](t-a_k(t)\right)\mbox{.}
\end{multline}
One of the conditions of Lemma~\ref{thm:ivp} was that
$k>\|A\|_\infty+R$. Thus, if $\operatorname{Re}\mu\geq-R$ we have that
\begin{align}\label{eq:expest:case2}
  \left[\|A\|_\infty-\operatorname{Re}\mu\right](t-a_k(t))\leq 
  \left[\|A\|_\infty+R\right](t-a_k(t))\leq 
  \frac{1}{k}\left[\|A\|_\infty+R\right]< 1
\end{align}
since $0\leq t-a_k(t)\leq 1/k$ by definition of $a_k$. Inserting
\eqref{eq:expest:case2} into \eqref{eq:app:case2} the integral term in
\eqref{eq:varpextest} is bounded by $\exp(1)/k$ such that
\begin{equation}
  \label{eq:lkest:case2}
  \|L_k(\mu)\|_\infty\leq \frac{\exp(1)}{k}\|B\|_\infty
\end{equation}

Inserting \eqref{eq:lkest:case2}, the more pessimistic of the two
estimates \eqref{eq:lkest:case1} and \eqref{eq:lkest:case2} for both
cases, into the upper bound for $L_k$ we get
\begin{displaymath}
  \|L_k(\mu)\|_\infty
  \leq\exp(1)\frac{\|B\|_\infty}{k}\mbox{.}
\end{displaymath}
Condition \eqref{eq:kbound} on $k$ from Lemma~\ref{thm:ivp} (requiring
that $k>\|B\|_\infty\exp(1+R)$) implies that the norm of $zL_k$ is
less than $1$. Consequently, $ I -zL_k(\mu)$ is invertible such that
the fixed-point problem \eqref{eq:varpextmap} has a unique solution
for all tuples $v$.
{\hfill$\square$}

\section{Multiplicative perturbations}
\label{app:sec:mult}
In Lemma~\ref{thm:pccurves} we had a triplet
$(\omega_0,\gamma_0,\varphi_0)$ such that
\begin{displaymath}
  h(\i\omega_0,\exp(-\gamma_0-\i\varphi_0))=0\mbox{, \quad and\quad}
  \partial_2h(\i\omega_0,\exp(-\gamma_0-\i\varphi_0))\neq0
\end{displaymath}
($h$ was an analytic complex function in both arguments). One has a
regular local curve $\gamma(\omega)+\i\varphi(\omega)$ of complex numbers
near $\gamma_0+\i\omega_0$ satisfying
\begin{displaymath}
  h(\i\omega,\exp(-\gamma(\omega)-\i\varphi(\omega)))=0
\end{displaymath}
for all $\omega\approx\omega_0$. Then, Lemma~\ref{thm:pccurves}
claims, the regular root curve
$\gamma_\epsilon(\omega)+\i\varphi_\epsilon(\omega)$ of
\begin{equation}\label{eq:app:hpert}
  h(\i\omega+\epsilon\gamma,\exp(-\gamma-\i\varphi))=0\mbox{,}
\end{equation}
which exists for small $\epsilon$, has the form
\begin{equation}\label{eq:app:pert}
  \gamma_\epsilon(\omega)+\i\varphi_\epsilon(\omega)=
  \gamma(\omega)+\i\varphi(\omega)
  +\epsilon g_\epsilon(\omega)\gamma(\omega)\mbox{.}
\end{equation}
The emphasis in \eqref{eq:app:pert} is on the factor $\gamma(\omega)$
in the error term $\epsilon g_\epsilon(\omega)\gamma(\omega)$, which
comes from the special type of perturbation in
equation~\eqref{eq:app:hpert} defining the curve. Note that
$g_\epsilon(\omega)$ is complex.

This fact is a special case of the following general statement:
\begin{lemma}\label{thm:multpert}
  Let $\epsilon$ be small, the function 
  $f:\mathbb{R}^n\times\mathbb{R}^m\mapsto\mathbb{R}^m$
  be smooth, $f(x_*,y_*)=0$ and $\partial_2f(x_*,y_*)$ be
  invertible. Let $A\in\mathbb{R}^{n\times m}$ be a matrix. Then the curve
  $y_\epsilon(x)$ defined implicitly by
  \begin{equation}\label{eq:app:f}
    f(x+\epsilon Ay_\epsilon,y_\epsilon)=0
  \end{equation}
  for $x\approx x_*$ has the form
  \begin{equation}\label{eq:app:multpert}
    y_\epsilon(x)=\left[ I +\epsilon g(x,\epsilon)A\right]y_0(x)
  \end{equation}
  where $g$ is a $m\times n$ matrix depending smoothly on $x$ and $\epsilon$.
\end{lemma}
Note that $y_0(x)$ is the curve defined implicitly by $f(x,y_0)=0$
(putting $\epsilon=0$ in \eqref{eq:app:f}).

\begin{proof}
  The Implicit Function Theorem guarantees that $y_\epsilon(x)$ exists
  for small $\epsilon$ and $x\approx x_*$, and that it has the form
  $y_\epsilon(x)=y_0(x)+\epsilon h(x,\epsilon)$. Subtracting the
  expressions $f(x+\epsilon Ay_\epsilon,y_\epsilon)$ and
  $f(x,y_0(x))$, which are both zero, from each other, and applying
  the mean value theorem we obtain
  \begin{equation}\label{eq:app:avg}
    0=\epsilon D_1(x,\epsilon)\,Ay_\epsilon+
    D_2(x,\epsilon)[y_\epsilon-y_0(x)]
  \end{equation}
  where $D_1$ and $D_2$ are the averaged derivatives:
  \begin{align*}
    D_1&=\int_0^1\partial_1f(x+ s\epsilon A(y_0(x)+\epsilon
    h(x,\epsilon)),
    y_0(x)+s\epsilon h(x,\epsilon)) \,{\mathrm{d}} s\\
    D_2&=\int_0^1\partial_2f(x+ s\epsilon A(y_0(x)+\epsilon
    h(x,\epsilon)), y_0(x)+s\epsilon h(x,\epsilon)) \,{\mathrm{d}} s
  \end{align*}
  Note that we have replaced $y_\epsilon$ by $y_0+\epsilon
  h(x,\epsilon)$ inside the arguments of $D_1$ and $D_2$. Since $D_2$
  is invertible for small $\epsilon$ and $x\approx x_*$ we can
  rearrange \eqref{eq:app:avg} for $y_\epsilon$ (dropping the
  arguments $x$ and $\epsilon$ from $D_1$ and $D_2$:
  \begin{align}
    y_\epsilon&=
    \left[ I +\epsilon D_2^{-1}D_1\,A\right]^{-1}y_0(x)\nonumber\\
    &=\left[ I -\epsilon\left( I + \epsilon D_2^{-1}D_1\,A\right)^{-1}
      D_2^{-1}D_1\,A\right]y_0(x)\mbox{,}\label{eq:app:yepsform}
  \end{align}
  which is of the form \eqref{eq:app:multpert} as claimed by the
  lemma.
\end{proof}


Note that \eqref{eq:app:yepsform} is not an explicit definition of
$y_\epsilon(x)$ but rather a fixed point problem for $y_\epsilon$
because $y_\epsilon$ occurs on the right-hand side as well (via the
unknown function $h$). However, the Banach contraction Mapping
Principle can be applied to the fixed point problem
\eqref{eq:app:yepsform} to produce an explicit definition of
$y_\epsilon$.

If we treat the complex numbers as a two-dimensional vector space then
$\omega$ (treated as complex number) plays the role of
$x\in\mathbb{R}^2$, $\gamma+\i\varphi\in\mathbb{C}=\mathbb{R}^2$ plays
the role of $y$, $(x,y)\mapsto h(\i x,\exp(-y))$ plays the role of $f$,
and $y\mapsto-\i\operatorname{Re} y$ plays the role of $A$.
\end{document}